\documentclass[11pt]{article}
\usepackage{tikz}
\usepackage{endnotes}
\usepackage[hyperindex,breaklinks]{hyperref}

\usepackage{hyperendnotes}
\usepackage{float}
\usepackage{latexsym,calligra,amsmath,amssymb,stmaryrd,MnSymbol,ushort,amsthm}
\usepackage[parfill]{parskip} 
\usepackage{multicol}
\usepackage[bf, small, justification=justified,singlelinecheck=false]{caption}
\newif\ifdviwin
\usepackage{ulem}
\usepackage[all,cmtip]{xy}
\usepackage[latin1]{inputenc}
\usepackage[english]{babel}
\usepackage{indentfirst}
\usepackage[mathscr]{eucal}
\usepackage{amssymb,amsmath,amsfonts}
\usepackage{graphicx}
\usepackage{amsmath,amscd}
\usepackage{xxcolor}
\newtheorem{theorem}{Theorem}
\newtheorem{lem}[theorem]{Lemma}

\newtheorem{defi}[theorem]{Definition}

\newtheorem{example}[theorem]{Example}

\newtheorem*{theorem*}{Main Theorem}
\numberwithin{theorem}{subsection} 
\numberwithin{equation}{subsection} 
\DeclareMathAlphabet{\mathcalligra}{T1}{calligra}{m}{n}
\DeclareFontShape{T1}{calligra}{m}{n}{<->s*[2.2]callig15}{}

\newcommand{\ms}[1]{\mathscr{#1}}
\newcommand{\mbf}[1]{\mathbf{#1}}
\newcommand{\mbb}[1]{\mathbb{#1}}
\newcommand{\mfr}[1]{\mathfrak{#1}}

\newcommand{\tn}[1]{\textnormal{#1}}
\newcommand{\fsz}[1]{\footnotesize{#1}}
\def\ee{\varepsilon}

\def\QQ{\mathbb{Q}}
\def\ZZ{\mathbb{Z}}

\tikzset{isometricXYZ/.style={x={(-0.707cm,-0.354cm)}, y={(0.707cm,-0.354cm)}, z={(0cm,1cm)}}}
\newcommand{\comment}[1]{}
\begin{document}

\title{A Goodwillie-type Theorem for Milnor $K$-Theory}
\author{Benjamin F. Dribus \\ \small{Louisiana State University} \\ \small{{\color{blue} bdribus@math.lsu.edu}}}
\maketitle

\begin{abstract} \noindent Goodwillie's rational isomorphism between relative algebraic $K$-theory and relative cyclic homology of a ring with respect to a nilpotent ideal, together with the $\lambda$-decomposition of cyclic homology, illustrates the close relationships among algebraic $K$-theory, cyclic homology, and differential forms.  In this paper, I prove a Goodwillie-type theorem for relative Milnor $K$-theory, working over a very general class of commutative rings, defined via the stability criterion of Van der Kallen.  The theorem expresses relative Milnor $K$-theory exactly, rather than merely rationally, in terms of absolute K\"{a}hler differentials.  Early results of Van der Kallen and Bloch, involving $K_2$, are special cases.   The version of Milnor $K$-theory used is the na\"{i}ve one, defined in terms of tensor algebras, which generalizes Milnor's original definition for fields.  The result likely generalizes in terms of de Rahm-Witt complexes by weakening some invertibility assumptions, but the class of rings considered is already more than sufficiently general for the intended applications.  The main motivation for this paper arises from applications to the infinitesimal theory of Chow groups, first pointed out by Bloch in the 1970's, and prominent in recent work of Green and Griffiths.  In this context, much can be accomplished geometrically without much $K$-theoretic sophistication, although the proper structural viewpoint really involves ``deep, modern" methods and results such as Thomason's localization theorem.  Milnor $K$-theory, by contrast, is a simple type of ``symbolic $K$-theory," meaning that it may be defined concretely in terms of group presentations, rather than requiring the more sophisticated homotopy-theoretic constructions of Quillen, Waldhausen, Bass, and Thomason.  From this viewpoint, the theorem in this paper is ``not very $K$-theoretic" in a modern sense, but is merely the answer to a particular group isomorphism problem.  Early $K$-theory is replete with such problems, often involving messy, ad hoc proofs.   The proof here is of a similar character, involving elementary but intricate symbolic manipulations.  The advantage of this approach is that it is amenable to straightforward calculations, which modern $K$-theory often is not.  The proof is by induction, beginning with Bloch's result for $K_2$.  Related results and geometric applications are discussed in the final section. 
\end{abstract}

\section{Introduction}\label{sectionintroduction}

\subsection{Statement of the Theorem}\label{subsectionstatement}

Goodwillie's isomorphism \cite{GoodWillieRelativeK86}: 
\begin{equation}\label{equgoodwillie}
K_{n+1}(R,I)\otimes\mathbb{Q}\cong HC_{n}(R,I)\otimes\mathbb{Q},
\end{equation}
relating the relative algebraic $K$-theory and relative cyclic homology of a ring $R$ with respect to a $2$-sided nilpotent ideal $I$, together with the $\lambda$-decomposition of cyclic homology, which takes the form \cite{LodayCyclicHomology98}: 
\begin{equation}\label{equlambda}
HC_n(R;k)\cong \frac{\Omega_{R/k} ^n}{d\Omega_{R/k} ^{n-1}}\oplus H_{dR} ^{n-2}(R;k)\oplus H_{dR} ^{n-4}(R;k)\oplus...,
\end{equation}
 for a smooth algebra $R$ over a commutative ring $k$ containing $\QQ$, 
highlight the relationships among algebraic $K$-theory, cyclic homology, and differential forms.\footnotemark\footnotetext{I have chosen notation and context similar to that of Loday \cite{LodayCyclicHomology98} and Weibel \cite{WeibelKBook}.  Goodwillie \cite{GoodWillieRelativeK86} works in the more general context of {\it simplicial rings,} and uses $K_n(f)$ and $HC_n(f)$ to denote the relative groups $K_{n-1}(R,I)$ and $HC_{n-1}(R,I)$, where $f$ is the canonical surjection $R\rightarrow R/I$.  Note the difference of index conventions: Goodwillie (page 359) defines $K_n(f)$ to be the $(n-1)$st homotopy group of the homotopy fiber of the morphism $\mbf{K}(R)\rightarrow \mbf{K}(R/I)$ of pointed simplicial sets or spectra, while Loday (also page 359) and Weibel (Chapter IV, page 8) define $K_{n}(R,I)$ to be the $n$th homotopy group of the analogous fiber.  This leads to different numberings in the long exact sequences of Goodwillie (remark 3, page 359) and Loday (11.2.19.2, page 359).}   Goodwillie's isomorphism is a {\it relative} example of a rational isomorphism between an algebraic $K$-theory and a cohomology theory (compare \cite{FriedlanderRational03}); i.e., an isomorphism after tensoring both objects with $\QQ$. The first summand $\Omega_{R/k} ^n/d\Omega_{R/k} ^{n-1}$ appearing in equation \hyperref[equlambda]{\ref{equlambda}} is the $n$th module of K\"{a}hler differentials of $R$ relative to $k$, modulo exact differentials.  It is roughly analogous to the $(n+1)$st Milnor $K$-theory group $K_{n+1}^{\tn{\fsz{M}}}(R)$, which maps canonically into the first summand of the corresponding $\lambda$-decomposition of the algebraic $K$-theory group $K_{n+1}(R)$.\footnotemark\footnotetext{Here, $K_{n+1}(R)$ may be taken to be Quillen $K$-theory.  The map $K_{n+1}^{\tn{\fsz{M}}}(R)\rightarrow K_{n+1}(R)$ sends the Steinberg symbol $\{r_0,r_1,...,r_n\}$ to the product $r_0\times r_1\times...\times r_{n}$, where each factor $r_j$ in the latter product is regarded as an element of $K_1(R)\cong R^*$, and where the multiplication $\times$ is in the ring $K(R)$.   See Weibel  \cite{WeibelKBook} Chapter IV, pages 7-8, for details.  This map is not injective in general, even for fields.  See, for example, Weibel \cite{WeibelKBook} Chaper IV, exercise 1.12.}  The remaining summands $H_{\tn{\fsz{dR}}} ^{n-2j}(R;k)$ are de Rham cohomology modules; i.e., the cohomology modules of the algebraic de Rham complex $(\Omega_{R/k}^\bullet,d)$.  

In this paper I prove the following Goodwillie-type theorem for relative Milnor $K$-theory in the context of commutative rings: 
\vspace*{.2cm}
\begin{theorem*}Suppose that $R$ is a split nilpotent extension of a $5$-fold stable ring $S$, with extension ideal $I$, 
whose index of nilpotency is $N$.  Suppose further that every positive integer less than or equal to $N$
 is invertible in $S$.  Then for every positive integer $n$,
\begin{equation}\label{equmaintheorem}K_{n+1} ^{\tn{\fsz{M}}}(R,I)\cong \frac{\Omega_{R,I} ^n}{d\Omega_{R,I} ^{n-1}}.\end{equation}
\end{theorem*}

Here, $R$ and $S$ are commutative rings with identity, and $K_{n+1} ^{\tn{\fsz{M}}}(R,I)$ is the $(n+1)$st Milnor $K$-group of $R$ relative to $I$.  This group is what Kerz \cite{KerzMilnorLocal} calls the ``na\"{i}ve Milnor $K$-group;" see section \hyperref[subsectionsymbolic]{\ref{subsectionsymbolic}} below for more details. The differentials are {\it absolute} K\"{a}hler differentials, in the sense that they are differentials with respect to $\ZZ$.  They are {\it relative} to $I$ in the same sense that the $K$-groups are relative to $I$.  Because $R$ is a split extension of $S$, the group $K_{n+1} ^{\tn{\fsz{M}}}(R,I)$ may be identified with the kernel $\tn{Ker}[K_{n+1} ^{\tn{\fsz{M}}}(R)\rightarrow K_{n+1} ^{\tn{\fsz{M}}}(S)]$, and the group $\Omega_{R,I} ^n$ may be identified with the kernel $\tn{Ker}[\Omega_{R/\mathbb{Z}} ^n\rightarrow \Omega_{S/\mathbb{Z}} ^n]$, where both maps are induced by the split surjection $R\rightarrow S$.  The class of $m$-fold stable rings, defined in section \hyperref[subsectionstability]{\ref{subsectionstability}} below, includes any local ring with at least $m+2$ elements in its residue field. 

The isomorphism \hyperref[equmaintheorem]{\ref{equmaintheorem}} is the map
\[\phi_{n+1}:K_{n+1} ^{\tn{\fsz{M}}}(R,I)\longrightarrow\frac{\Omega_{R,I} ^n}{d\Omega_{R,I} ^{n-1}}\]
\begin{equation}\label{equationphi}
\{r_0,r_1,...,r_n\}\mapsto\log(r_0)\frac{dr_1}{r_1}\wedge...\wedge\frac{dr_n}{r_n},
\end{equation}
where $r_0$ belongs to the subgroup $(1+I)^*$ of the multiplicative group $R^*$ of $R$, and $r_1,...,r_n$ belong to $R^*$.  The symbol $\{r_0,r_1,...,r_n\}$ is called a Steinberg symbol; such symbols generate the Milnor $K$-group $K_{n+1} ^{\tn{\fsz{M}}}(R,I)$, as explained in section \hyperref[subsectionsymbolic]{\ref{subsectionsymbolic}} below.   The logarithm is understood in the sense of power series.  The inverse isomorphism is the map
\[\psi_{n+1}:\frac{\Omega_{R,I} ^n}{d\Omega_{R,I} ^{n-1}}\longrightarrow K_{n+1} ^{\tn{\fsz{M}}}(R,I)\]
\begin{equation}\label{equationpsi}
r_0dr_1\wedge...\wedge dr_m\wedge dr_{m+1}\wedge...\wedge dr_n\mapsto\{e^{r_0r_{m+1}...r_{n}},e^{r_1},...,e^{r_m},r_{m+1},...,r_{n}\},
\end{equation}
where the elements $r_0,r_1,...,r_m$ belong to the ideal $I$, and the elements $r_{m+1},...,r_{n}$ belong to the multiplicative group $R^*$ of $R$.   Part of the proof of the theorem involves verifying that the formulae \hyperref[equationphi]{\ref{equationphi}} and \hyperref[equationpsi]{\ref{equationpsi}}, defined in terms of special group elements, extend to homomorphisms. 


\subsection{Structure of the Paper}\label{subsectionstructure}

{\bf Section \hyperref[sectionprelim]{\ref{sectionprelim}}} provides context and background involving split nilpotent extensions, Van der Kallen stability, symbolic $K$-theory, and K\"{a}hler differentials.  

\begin{itemize}
\item Section \hyperref[subsectioncontextual]{\ref{subsectioncontextual}} places the theorem in its proper historical and mathematical context. 
\item Section \hyperref[subsectionnilpotent]{\ref{subsectionnilpotent}} introduces split nilpotent extensions of rings. 
\item Section \hyperref[subsectionstability]{\ref{subsectionstability}} discusses Van der Kallen's stability criterion for rings, and includes an easy lemma on the behavior of stability under split nilpotent extensions. 
\item Section \hyperref[subsectionsymbolic]{\ref{subsectionsymbolic}} introduces Milnor $K$-theory and Dennis-Stein $K$-theory, two different symbolic $K$-theories.   The definition of Milnor $K$-theory used is the na\"{i}ve one in terms of tensor algebras.   Dennis-Stein $K$-theory is defined here only for $K_2$.   Brief historical context is provided, and different definitions appearing in the literature are mentioned.  In particular, Kerz's ``improved Milnor $K$-theory," and Thomason's nonexistence proof for ``ideal global Milnor $K$-theory," are cited. 
\item Section \hyperref[subsectionsymbolicstability]{\ref{subsectionsymbolicstability}} presents two useful existing results relating symbolic $K$-theories under stability hypotheses.  The first, theorem \hyperref[theoremvanderkallen]{\ref{theoremvanderkallen}}, is Van der Kallen's isomorphism between the second Milnor $K$-theory group and the second Dennis-Stein $K$-theory group in the $5$-fold stable case.  The second, theorem \hyperref[theoremmaazen]{\ref{theoremmaazen}}, is Maazen and Stienstra's isomorphism between relative $K_2$ and the corresponding second relative Dennis-Stein group of a split radical extension. 
\item Section \hyperref[subsectionkahlerbloch]{\ref{subsectionkahlerbloch}} introduces absolute K\"{a}hler differentials, and cites a famous result of Bloch, theorem \hyperref[theorembloch]{\ref{theorembloch}}, which expresses relative $K_2$ of a split nilpotent extension in terms of absolute K\"{a}hler differentials.   This theorem provides the base case of the main theorem in section \hyperref[subsectionbasecase]{\ref{subsectionbasecase}}.
\end{itemize} 

{\bf Section \hyperref[sectionSteinbergKahler]{\ref{sectionSteinbergKahler}}} supplies computational tools for working with Milnor $K$-theory and K\"{a}hler differentials.  

\begin{itemize}
\item Section \hyperref[subsectionnotation]{\ref{subsectionnotation}} fixes notation and conventions designed to streamline the proof in section \hyperref[sectionproof]{\ref{sectionproof}}.  This devices are specific to this paper, although they could be used to advantage in any similar context. 
\item Section \hyperref[generatorsMilnor]{\ref{generatorsMilnor}} discusses Milnor $K$-theory and relative Milnor $K$-theory in terms of generators and relations.  In particular, lemmas \hyperref[lemrelationsstable]{\ref{lemrelationsstable}} and \hyperref[lemrelativegenerators]{\ref{lemrelativegenerators}} establish the computational convenience of split nilpotent extensions of $5$-fold stable rings in this context. 
\item Section \hyperref[subsectiongeneratorsKahler]{\ref{subsectiongeneratorsKahler}} provides similar results for K\"{a}hler differentials.    
\item Section \hyperref[subsectiondlogdeRhamWitt]{\ref{subsectiondlogdeRhamWitt}} introduces the canonical $d\log$ map, which is a homomorphism of graded rings between Milnor $K$-theory and the absolute K\"{a}hler differentials.  The $d\log$ map plays a specific role in the proof of lemma \hyperref[lempatchingphi]{\ref{lempatchingphi}}.  This section also mentioned the de Rham Witt viewpoint, as suggested by Van der Kallen and Hesselholt.  
\end{itemize} 

{\bf Section \hyperref[sectionproof]{\ref{sectionproof}}} presents the proof of the theorem. 

\begin{itemize}
\item Section \hyperref[subsectionstrategy]{\ref{subsectionstrategy}} outlines the strategy of proof: ``induction and patching."  This section also explains why na\"{i}ve ``simpler" approaches seem to run into trouble. 
\item Section \hyperref[subsectionbasecase]{\ref{subsectionbasecase}} presents the base case of the theorem: $K_{2} ^{\tn{\fsz{M}}}(R,I)\cong \Omega_{R,I} ^1/dI$, proved by combining theorems \hyperref[theoremvanderkallen]{\ref{theoremvanderkallen}}, \hyperref[theoremmaazen]{\ref{theoremmaazen}}, and \hyperref[theorembloch]{\ref{theorembloch}}.   The isomorphisms in both directions are described explicitly. 
\item Section \hyperref[subsectionexplicitinduction]{\ref{subsectionexplicitinduction}} states the induction hypothesis in detail. 
\item Section \hyperref[subsectionanalysisphi]{\ref{subsectionanalysisphi}} gives the construction of the map $\displaystyle\phi_{n+1}: K_{n+1} ^{\tn{\fsz{M}}}(R,I)\rightarrow \Omega_{R,I} ^n/d\Omega_{R,I} ^{n-1}$, and the proof that it is a surjective homomorphism.  The strategy is to ``patch together" maps $\Phi_{n+1,j}$ for $0\le j\le n-1$.  
\item Section \hyperref[subsectionanalysispsi]{\ref{subsectionanalysispsi}} completes the proof of the theorem by giving the construction of the map $\displaystyle\psi_{n+1}:\Omega_{R,I} ^n/d\Omega_{R,I} ^{n-1}\rightarrow K_{n+1} ^{\tn{\fsz{M}}}(R,I)$, and the proof that $\phi_{n+1}$ and $\psi_{n+1}$ are inverse isomorphisms. 
\end{itemize} 

{\bf Section \hyperref[sectiondiscussion]{\ref{sectiondiscussion}}} includes discussion and applications of the theorem.  

\begin{itemize}
\item Section \hyperref[subsectionGG]{\ref{subsectionGG}} discusses the initial motivation for the paper, which is the recent work by Green and Griffiths on the infinitesimal structure of cycle groups and Chow groups.  
\item Section \hyperref[subsectionsimilarresults]{\ref{subsectionsimilarresults}} examines existing results concerning relative $K$-theory which are similar to the theorem in this paper.  First, an early special case of the theorem, due to Van der Kallen, leads to an expression for the tangent group to the second Chow group $\tn{Ch}^2(X)$ of a smooth projective surface $X$ defined over a field containing the rational numbers.  This result, spelled out in equation \hyperref[linearblochstheorem]{\ref{linearblochstheorem}}, plays a prominent role in the work of Green and Griffiths.   Second, Stienstra carries this line of reasoning further to define the formal completion of $\tn{Ch}^2(X)$, essentially by sheafifying equation Bloch's theorem \hyperref[theorembloch]{\ref{theorembloch}} for relative $K_2$.  The resulting expression appears in equation \hyperref[equstienstra]{\ref{equstienstra}}.  Third, Hesselholt has proven an analogous result for the relative $K$-theory (not just Milnor $K$-theory) of a truncated polynomial algebra.   This result appears in equation \hyperref[equhesselholt]{\ref{equhesselholt}}. The first summand on the right-hand side gives a special case of the theorem in this paper, under appropriate assumptions on the underlying ring.    
\item Section \hyperref[subsectiontangentfunctors]{\ref{subsectiontangentfunctors}} discusses various ways of generalizing Green and Griffiths' tangent functors in a geometric context. 
\end{itemize}

\section{Preliminaries}\label{sectionprelim}

\subsection{Contextual Remarks}\label{subsectioncontextual}

From an abstract viewpoint, the mathematical problem addressed by this paper is a {\it group isomorphism problem,} in which one attempts to determine whether or not two groups, expressed in terms of generators and relations, are isomorphic.\footnotemark\footnotetext{Such problems were first studied systematically in the context of finite groups by Max Dehn more than a century ago.   In the present context, the groups are generally infinite, although some finite examples are included; for instance, those involving nilpotent extensions of finite fields.  The general group isomorphism is known to be undecidable, in the sense that no algorithm exists that will solve every case of the problem.} The early algebraic $K$-theory of the 1960's and 1970's provides many examples of such problems, often accompanied by forbidding symbolic computations.   The papers of Maazen and Stienstra \cite{MaazenStienstra77} and Van der Kallen \cite{VanderKallenRingswithManyUnits77} are representative.  The arguments in this paper follow this tradition; in particular, they will not stagger anyone with their beauty.  Perhaps the best justification for inflicting such material on the reader forty years after the papers mentioned above is that some people still wish, with justification, to carry out explicit elementary calculations in algebraic $K$-theory.   Here I have in mind particularly the recent work of Green and Griffiths \cite{GreenGriffithsTangentSpaces05} on the infinitesimal structure of cycle groups and Chow groups, in which the authors employ ``low-tech" wrangling with Steinberg symbols and K\"{a}hler differentials to achieve surprising geometric insights.   Such a viewpoint would be impossible without early results such as Matsumoto's theorem, which support, in special cases of particular interest, a na\"{i}ve symbolic treatment of $K$-theoretic structure possessing much greater intrinsic subtlety in the general case.  It is also true that relatively old and utilitarian methods can sometimes pick up crumbs left behind by the great machines of modern $K$-theory.  For example, the stability criterion of Van der Kallen, used in the theorem in this paper, provides a sharper result than the hypotheses that appear in many similar but more sophisticated theorems. 


\subsection{Split Nilpotent Extensions}\label{subsectionnilpotent}

A split nilpotent extension of a ring $S$ provides an algebraic notion of ``infinitesimal thickening of $S$."  Heuristically, one may think of augmenting $S$ by the addition of elements ``sufficiently small" that products of sufficiently many such elements vanish.\footnotemark\footnotetext{More precisely, one thinks of ``thickening" the affine scheme $\tn{Spec}(S)$ corresponding to $S$.   The motivation for mentioning this viewpoint comes from the geometric applications discussed in section \hyperref[sectiondiscussion]{\ref{sectiondiscussion}}.}\\
\begin{defi} Let $S$ be a commutative ring with identity.  A {\bf split nilpotent extension} of $S$ is a split surjection $R\rightarrow S$ whose kernel $I$, called the extension ideal, is nilpotent.  
\end{defi}
The {\it index of nilpotency} of $I$ is the smallest integer $N$ such that $I^N=0$.  A nilpotent extension ideal $I$ is contained in any maximal ideal $J$ of $R$, since $R/J$ is a field, and hence belongs to the Jacobson radical of $R$.  Hence, a split nilpotent extension is a special case of what Maazen and Stienstra \cite{MaazenStienstra77} call a {\it split radical extension.}\\

\begin{example}\tn{The simplest nontrivial split nilpotent extension of $S$ is the extension $R=S[\ee]/\ee^2\rightarrow S$.  The ring $S[\ee]/\ee^2$ is called the ring of dual numbers over $S$.  This is the extension involved in Van der Kallen's early computation \cite{VanderKallenEarlyTK271} of relative $K_2$, which plays a prominent role in the work of Green and Griffiths \cite{GreenGriffithsTangentSpaces05} on the infinitesimal structure of cycle groups and Chow groups.  More generally, if $k$ is a field and $S$ is a $k$-algebra, then tensoring $S$ with any local artinian $k$-algebra $A$ induces a split nilpotent extension $S\otimes_k A\rightarrow S$.  These are the extensions considered by Stienstra \cite{StienstraFormalCompletion83} in his study of the formal completion of the second Chow group of a smooth projective surface over a field containing the rational numbers.}
\end{example}



\subsection{Van der Kallen Stability}\label{subsectionstability}

Certain convenient properties of algebraic $K$-theory, including those enabling some of the steps of the proof in section \hyperref[sectionproof]{\ref{sectionproof}} below, rely on an assumption that the ring under consideration has ``enough units," or that its units are ``organized in a convenient way."   One way to make this idea precise is via Van der Kallen's \cite{VanderKallenRingswithManyUnits77} notion of stability.\footnotemark\footnotetext{This terminology seems to have first appeared in Van der Kallen, Maazen, and Stienstra's 1975 paper {\it A Presentation for Some $K_2(R,n)$} \cite{VanMaazenStienstra}.}  This notion is closely related to the stable range conditions of Hyman Bass, introduced in the early 1960's.\\
\begin{defi}Let $S$ be a  commutative ring with identity, and let $m$ be a positive integer.  
\begin{enumerate}
\item A pair $(s,s')$ of elements of $S$ is called {\bf unimodular} if $sS+s'S=S$. 
\item  $S$ is called $\mbf{m}$-{\bf fold stable} if, given any family $\{(s_j,s_j')\}_{j=1}^m$ of unimodular pairs in $S$, there exists an element $s\in S$ such that $s_j+s_j's$ is a unit in $S$ for each $j$. 
\end{enumerate}
\end{defi}
\vspace*{.3cm}
\begin{example}\label{examplesemilocalstable}\tn{A semilocal ring\footnotemark\footnotetext{A {\it commutative} ring $R$ is semilocal if and and only if it has a finite number of maximal ideals.  The general definition is that $R/J(R)$ is semisimple, where $J(R)$ is the Jacobson radical of $R$.} is $m$-fold stable if and only if all its residue fields contain at least $m+1$ elements.  See Van der Kallen, Maazen and Stienstra \cite{VanMaazenStienstra}, page 935, or Van der Kallen \cite{VanderKallenRingswithManyUnits77}, page 489.  In particular, for any $m$, the class of $m$-fold stable rings is much larger than the class of local rings of smooth algebraic varieties over a field containing the rational numbers, which are the rings of principal interest in the context of Green and Griffiths' work on the infinitesimal structure of cycle groups and Chow groups \cite{GreenGriffithsTangentSpaces05}.  Due to the relationship between stability and the size of residue fields, the theorem in this paper allows some of the computations of Green and Griffiths to be repeated in positive characteristic.}
\end{example}

The following easy lemma establishes two consequences of stability necessary for the proof of the theorem in section \hyperref[sectionproof]{\ref{sectionproof}} below.\\


\begin{lem}\label{lemstability} Suppose $R$ is a split nilpotent extension of a commutative ring $S$ with identity.  Let $I$ be the extension ideal. 
\begin{enumerate} 
\item If $S$ is $m$-fold stable, then $R$ is also $m$-fold stable.
\item If $S$ is $2$-fold stable and $2$ is invertible in $S$, then every element of $R$ is the sum of 
two units.
\end{enumerate}
\end{lem}
\begin{proof} For part 1 of the lemma, let $\{(r_j,r_j')\}_{j=1}^m$ be a family of unimodular pairs in $R$.  Since $R$ is a split extension of $S$, $r_j$ and $r_j'$ may be written uniquely as sums
\[r_j=s_j+i_j,\hspace*{.5cm}\tn{and}\hspace*{.5cm}r_j'=s_j'+i_j',\hspace*{.5cm}\tn{where}\hspace*{.5cm}s_j,s_j'\in S\hspace*{.5cm}\tn{and}\hspace*{.5cm}i_j,i_j'\in I.\]
It follow immediately that $\{(s_j,s_j')\}_{j=1}^m$ is a family of unimodular pairs in $S$.  Since $S$ is $m$-fold stable, there exists an element $s$ of $S$ such that $s_j+s_j's$ is a unit in $S$ for each $j$.   Then $r_j+r_j's$ may be expressed as a sum of a unit and a nilpotent element as follows:
\[r_j+r_j's=(s_j+s_j's)+(i_j+i_j's).\]
Therefore, $r_j+r_j's$ is a unit in $R$.   This is true for every $j$, so $R$ is $m$-fold stable. 

For part 2 of the lemma, first note that if $S$ is $2$-fold stable and $2$ is invertible in $S$, the same hypotheses hold for $R$ by part 1 of the lemma.  Let $r$ be any element of $R$. Any pair of elements including a unit is automatically unimodular, so the pairs $(r,2)$ and $(r,-2)$ are unimodular (these pairs need not be distinct).  Since $R$ is $2$-fold stable, there exists an element $r'$ in $R$, and units $u$ and $v$ in $R$, such that 
 \[r+2r'=u\hspace*{.5cm}\tn{and}\hspace*{.5cm} r-2r'=v.\]
 Adding these formulas gives $2r=u+v$.  Since $2$ is invertible in $R$, this implies that $r=u/2+v/2$, a sum of two units. 
\end{proof}

\subsection{Symbolic $K$-Theories: Milnor $K$-Theory and Dennis-Stein $K$-Theory}\label{subsectionsymbolic}

Milnor $K$-theory and Dennis-Stein $K$-theory are {\it symbolic $K$-theories;} an informal term which means, in this context, $K$-theories whose $K$-groups admit simple presentations via generators, called {\it symbols,} and relations.  More sophisticated ``modern" $K$-theories, such as Quillen's $K$-theory, Waldhausen's $K$-theory, and the amplifications of Bass and Thomason, have homotopy-theoretic definitions.  Symbolic $K$-theories have the advantage of being relatively elementary, but tend to lack certain desirable formal properties.  In this sense, they represent one extreme of the seemingly unavoidable tradeoff between accessibility and formal integrity in algebraic $K$-theory.

The following definition introduces the ``na\"{i}vest" version of Milnor $K$-theory:\\    

\begin{defi}\label{defiMilnorK} Let $R$ be a commutative ring with identity, and let $R^*$ be its multiplicative group of invertible elements, viewed as a $\ZZ$-module. 
\begin{enumerate}
\item The {\bf Milnor $K$-ring} $K_\bullet^{\tn{\fsz{M}}}(R)$\footnotemark\footnotetext{The ``dot notation" $K_\bullet^{\tn{\fsz{M}}}(R)$, rather than the simpler $K^{\tn{\fsz{M}}}(R)$, is used here for the purposes of comparing the Milnor $K$-ring to the graded ring $\Omega_{R/\ZZ}^\bullet$ of absolute K\"{a}hler differentials in lemma \hyperref[lemdlog]{\ref{lemdlog}} below, since $\Omega_{R/\ZZ}$ always means $\Omega_{R/\ZZ}^1$.} of $R$ is the quotient
\[K_\bullet^{\tn{\fsz{M}}}(R):=\frac{T_{R^*/\mathbb{Z}}}{I_{\tn{\fsz{St}}}}\]
of the tensor algebra $T_{R^*/\mathbb{Z}}$ by the ideal $I_{\tn{\fsz{St}}}$ generated by elements of the form $r\otimes(1-r)$.  
\item The $n$th {\bf Milnor $K$-group} $K_{n} ^{\tn{\fsz{M}}}(R)$ of $R$, defined for $n\ge0$, is the $n$th graded piece of $K_\bullet^{\tn{\fsz{M}}}(R)$. 
\end{enumerate}
\end{defi}
The tensor algebra $T_{R^*/\mathbb{Z}}$ of $R$ over $k$ is by definition the graded $k$-algebra whose zeroth graded piece is $k$, whose $n$th graded piece is the $n$-fold tensor product $R\otimes_k...\otimes_kR$ for $n\ge1$, and whose multiplicative operation is induced by the tensor product.  The subscript ``$\tn{St}$" assigned to the ideal $I_{\tn{\fsz{St}}}$ stands for ``Steinberg," since the defining relations $r\otimes(1-r)\sim0$ of $K^{\tn{\fsz{M}}}(R)$ are called {\it Steinberg relations}.  The ring $K^{\tn{\fsz{M}}}(R)$ is noncommutative, since concatenation of tensor products is noncommutative; more specifically, it is {\it anticommutative} if $R$ has ``enough units," in a sense made precise below.  The $n$th Milnor $K$-group $K_{n} ^{\tn{\fsz{M}}}(R)$ is generated, under {\it addition} in $K^{\tn{\fsz{M}}}(R)$, by equivalence classes of $n$-fold tensors $r_1\otimes...\otimes r_n$.   Such equivalence classes are denoted by symbols $\{r_1,...,r_n\}$, called {\it Steinberg symbols}.   When working with individual Milnor $K$-groups, the operation is usually viewed {\it multiplicatively,} and the identity element is usually denoted by $1$.   For example, expressions such as $\prod_l\{r_l,e^{u_jr_li_k\Pi_l},\bar{r}_l,u_j\}$, appearing in sections \hyperref[subsectionanalysisphi]{\ref{subsectionanalysisphi}} and  \hyperref[subsectionanalysispsi]{\ref{subsectionanalysispsi}} below, are viewed as products in $K_{n} ^{\tn{\fsz{M}}}(R)$, although they represent {\it sums} in $K^{\tn{\fsz{M}}}(R)$. 

Milnor $K$-theory first appeared in John Milnor's 1970 paper {\it Algebraic $K$-Theory and Quadratic Forms} \cite{MilnorAlgebraicKTheoryQforms70}, in the context of fields.  Around the same time, Milnor, Steinberg, Matsumoto, Dennis, Stein, and others were studying the second $K$-group $K_2(R)$ of a general ring $R$, defined by Milnor in 1967 as the center of the Steinberg group of $R$.  $K_2(R)$ is often called ``Milnor's $K_2$" in honor of its discoverer, but is in fact the ``full $K_2$-group."  In particular, it is much more complicated in general than the second Milnor $K$-group $K_2^{\tn{\footnotesize{M}}}(R)$ according to definition \hyperref[defiMilnorK]{\ref{defiMilnorK}}.  Adding further to the confusion of terminology, the two groups $K_2(R)$ and $K_2^{\tn{\footnotesize{M}}}(R)$ {\it are} equal in many important special cases; in particular, when $R$ is a field, a division ring, local ring, or even a semilocal ring.\footnotemark\footnotetext{See \cite{WeibelKBook}, Chapter III, Theorem 5.10.5, page 43, for details.}  This result is usually called Matsumoto's theorem, since its original version was proved by Matsumoto, for fields, in an arithmetic setting.  Matsumoto's theorem was subsequently extended to division rings by Milnor, and finally to semilocal rings by Dennis and Stein. 

There is no consensus in the literature about how the Milnor $K$-groups $K_n^{\tn{\fsz{M}}}(R)$ should be defined for general $n$ and $R$.   The definition I use here is the most na\"{i}ve one.  Its claim to relevance relies on foundational work by Steinberg, Milnor, Matsumoto, and others.  Historically, the Steinberg symbol arose as a map $R^*\times R^*\rightarrow K_2(R)$, defined in terms of special matrices.   The properties of this map, including the relations satisfied by its images, may be analyzed concretely in terms of matrix properties.\footnotemark\footnotetext{See Weibel \cite{WeibelKBook} Chapter III, or Rosenberg \cite{RosenbergK94} Chapter 4 for details.}  In the case where $R$ is a field, Matsumoto's theorem states that the image of the Steinberg symbol map generates $K_2(R)$, and that all the relations satisfied by elements of the image follow from the relations of the tensor product and the Steinberg relations.   This allows a simple re-definition of $K_2(R)$ in terms of a tensor algebras when $R$ is a field, with the generators {\it renamed} Steinberg symbols.  Abstracting this result to general $n$ and $R$ leads to definition \hyperref[defiMilnorK]{\ref{defiMilnorK}} above.   However, it has been understood from the beginning that the resulting ``Milnor $K$-theory" is seriously deficient in many respects.   Quillen, Waldhausen, Bass, Thomason, and others have since addressed many of these deficiencies by defining more elaborate versions of $K$-theory, but there still remain many reasons why symbolic $K$-theories are of interest.  For example, they are closely connected to motivic cohomology, provide interesting approaches to the study of Chow groups and higher Chow groups, and arise in physical settings in superstring theory and elsewhere.  The viewpoint of the present paper, involving $\lambda$-decompositions, cyclic homology, and differential forms, is partly motivated by these considerations, particularly the theory of Chow groups. 

It is instructive to briefly examine a few different treatments of Milnor $K$-theory in the literature.  Weibel \cite{WeibelKBook} chooses to confine his definition of Milnor $K$-theory to the original context of fields (Chapter III, section 7), while defining Steinberg symbols more generally (Chapter IV, example 1.10.1, page 8), and also discussing many other types of symbols, including Dennis-Stein symbols (Chapter III, defnition 5.11, page 43), and Loday symbols (Chapter IV, exercise 1.22, page 122).\footnotemark\footnotetext{Interestingly, the Loday symbols project nontrivially into a range of different pieces of the $\lambda$-decomposition of Quillen $K$-theory.  See Weibel \cite{WeibelKBook} Chapter IV, example 5.11.1, page 52, for details.} Elbaz-Vincent and M\"{u}ller-Stach \cite{ElbazVincentMilnor02} define Milnor $K$-theory for general rings (Definition 1.1, page 180) in terms of generators and relations, but take the additive inverse relation of lemma \hyperref[lemrelationsstable]{\ref{lemrelationsstable}} as part of the definition.   The result is a generally nontrivial quotient of the Milnor $K$-theory of definition \hyperref[defiMilnorK]{\ref{defiMilnorK}} above.\\

\begin{example}\tn{Let $R=\ZZ_2[x]/x^2$.   The multiplicative group $R^*$ is isomorphic to $\ZZ_2$, generated by the element $r:=1+x$.  The Steinberg ideal is empty since $1-r$ is not a unit.  Hence, $K_2^{\tn{\fsz{M}}}(R)$ is just $R^*\otimes_\ZZ R^*\cong\ZZ_2$, generated by the symbol $\{r,r\}=\{r,-r\}$, while Elbaz-Vincent and M\"{u}ller-Stach's corresponding group is trivial.   By contrast, the additive inverse relation $\{r,-r\}=1$ always holds if one uses the original definition of Steinberg symbols in terms of matrices; see Weibel \cite{WeibelKBook} Chapter III, remark 5.10.4, page 43.  This may be interpreted as an indication that this relation is a desirable property for ``enhanced" versions of  Milnor $K$-theory.}
\end{example}

Moritz Kerz \cite{KerzMilnorLocal} has suggested an ``improved version of Milnor $K$-theory," motivated by a desire to correct certain formal shortcomings of the ``na\"{i}ve" version defined in terms of the tensor product.  For example, this version fails to satisfy the Gersten conjecture.   Thomason \cite{ThomasonNoMilnor92} has shown that Milnor $K$-theory does not extend to a theory of smooth algebraic varieties with desirable properties such as $\mbb{A}^1$-homotopy invariance and functorial homomorphisms to more complete version of $K$-theory. Hence, the proper choice of definition depends on what properties and applications one wishes to study. 


Dennis-Stein $K$-theory plays only on small part in this paper.  It therefore suffices to define only the ``second Dennis-Stein $K$-group."\footnotemark\footnotetext{Dennis and Stein initially considered a symbolic version of $K_2$ in their 1973 paper {\it $K_2$ of radical ideals and semi-local rings revisited} \cite{DennisStein}.  Their approach generalizes to higher $K$-theory in a variety of different ways; see, for instance, sections 11.1 and 11.2 of Loday \cite{LodayCyclicHomology98}.  What I mean by ``Dennis-Stein $K$-theory" is essentially the part of $K$-theory generated by the Loday symbols, but the distinction is immaterial for $K_2$.}\\ 
\begin{defi}\label{defidennisstein} Let $R$ be a commutative ring with identity, and let $R^*$ be its multiplicative group of invertible elements.  The second {\bf Dennis-Stein $K$-group} $D_2(R)$ of $R$ is the multiplicative abelian group whose generators are symbols $\langle a,b\rangle$ for each pair of elements $a$ and $b$ in $R$ such that $1+ab\in R^*$, subject to the additional relations
\begin{enumerate}
\item $\langle a,b\rangle\langle -b,-a\rangle=1.$
\item $\langle a,b\rangle\langle a,c\rangle=\langle a,b+c+abc\rangle.$
\item $\langle a,bc\rangle=\langle ab,c\rangle\langle ac,b\rangle.$
\end{enumerate}
\end{defi}

This definition may be found in both Maazen and Stienstra \cite{MaazenStienstra77} definition 2.2, page 275, and Van der Kallen \cite{VanderKallenRingswithManyUnits77}, page 488.


\subsection{Symbolic $K$-Theory and Stability}\label{subsectionsymbolicstability}

For rings possessing a sufficient degree of stability in the sense of Van der Kallen, different versions of symbolic $K$-theory tend to produce isomorphic $K$-groups. An important example of such an isomorphism involves the second Dennis-Stein $K$-group and the second Milnor $K$-group.\\
\begin{theorem}\label{theoremvanderkallen} (Van der Kallen) Let $S$ be a commutative ring with identity, and suppose that $S$ is $5$-fold stable.  Then 
\begin{equation}K_{2} ^{\tn{\fsz{M}}}(S)\cong D_2(S).\end{equation}
\end{theorem}
\begin{proof}Van der Kallen \cite{VanderKallenRingswithManyUnits77}, theorem 8.4, page 509.  Note that Van der Kallen denotes the Dennis-Stein group $D_2$ by $\tn{D}$, and the Milnor $K$-group $K_{2} ^{\tn{\fsz{M}}}$ by $\tn{US}$.  Definitions of the groups $\tn{D}(S)=D_2(S)$ and $\tn{US}(S)=K_{2} ^{\tn{\fsz{M}}}(S)$ in terms of generators and relations appears on pages 488 and 509 of the same paper, respectively.
\end{proof}
Whether or not theorem \hyperref[theoremvanderkallen]{\ref{theoremvanderkallen}} remains true if one weakens the stability hypothesis to $4$-fold stability apparently remains unknown.\footnotemark\footnotetext{Van der Kallen  \cite{VanderKallenRingswithManyUnits77} writes {\it ``We do not know if $4$-fold stability suffices for theorem 8.4."} More recently, Van der Kallen tells me \cite{VanderKallenprivate14} that the answer to this question is still apparently unknown.}

The following result involving {\it relative} $K$-groups, does not require any stability hypothesis.  Note that the group $K_{2}(R,I)$ is {\it a priori} the ``total relative $K$-group," not the just the part generated by Steinberg symbols.\\ 

\begin{theorem}\label{theoremmaazen} (Maazen and Stienstra) If $R$ is a split radical extension of $S$ with extension ideal $I$, then 
\begin{equation}K_{2}(R,I)\cong D_{2}(R,I)\end{equation}
\end{theorem}
\begin{proof}Maazen and Stienstra \cite{MaazenStienstra77}, theorem 3.1, page 279.
\end{proof}
In general, the relative $K$-groups $K_n(R,I)$ are defined so as to possess convenient functorial properties, and this does not always lead to a simple description in terms of $K_n(R)$ and $K_n(S)$.\footnotemark\footnotetext{In particular, these groups are often defined via homotopy fibers, as mentioned in the first footnote of section \hyperref[subsectionstatement]{\ref{subsectionstatement}}.  This guarantees the existence of a long exact sequence relating absolute and relative groups.  See Weibel \cite{WeibelKBook} Chapter IV, page 8, for details.}  For the purposes of this paper, however, the relative groups $K_n (R,I)$ may be identified with the kernels $\tn{Ker}[K_n (R)\rightarrow K_n (R/I)]$, and similarly for the Milnor and Dennis-Stein $K$-groups.  This is because the extension of $S$ by $I$ to obtain $R$ is assumed to be a {\it split} extension. 


\subsection{Absolute K\"{a}hler Differentials; a Result of Bloch}\label{subsectionkahlerbloch}

K\"{a}hler differentials provide a purely algebraic notion of differential forms in the context of commutative rings.  In the noncommutative context, differential forms are superseded by algebra cohomology theories.  Historically, expanding the role of differential forms was one of the primary motivations for the development of cyclic homology and cohomology.  This renders natural the appearance of cyclic homology in Goodwillie's isomorphism.\\  
\begin{defi}\label{defikahler}Let $R$ be a commutative $k$-algebra over a commutative ring $k$ with identity. The $k$-module of {\bf K\"{a}hler differentials} $\Omega^1_{R/k}$ of $R$ with respect to $k$ is the module generated over $k$ by symbols of the form $rdr'$, subject to the relations 
\begin{enumerate}
\item $rd(\alpha r'+\beta r'')=\alpha rdr'+\beta rdr''$ for $\alpha,\beta\in k$ and  $r,r',r''\in M$ ($k$-linearity).
\item $rd(r'r'')=rr'dr''+rr''dr'$ (Leibniz rule).
\end{enumerate}
The ring $\Omega_{R/k}^\bullet$ of K\"{a}hler differentials of $R$ with respect to $k$ is the exterior algebra over $\Omega^1_{R/k}$; i.e., the graded ring whose zeroth graded piece is $k$, whose $n$th graded piece is $\bigwedge^n\Omega^1_{R/k}:=\Omega^n_{R/k}$, and whose multiplication is wedge product. 
\end{defi}
The differential graded ring $(\Omega_{R/k}^\bullet,d)$, where the map $d$ takes the differential $r_0dr_1\wedge...\wedge dr_n$ to  the differential $dr_0\wedge dr_1\wedge...\wedge dr_n$, may be viewed as a complex, called the {\it algebraic de Rham complex.}  If the ground ring $k$ is the ring of integers $\ZZ$, then the modules $\Omega^n_{R/\ZZ}$ are abelian groups (i.e., $\ZZ$-modules), called the groups of {\bf absolute K\"{a}hler differentials.}  Groups $\Omega^n_{R,I}$ of absolute K\"{a}hler differentials {\it relative to an ideal} $I\subset R$ may be defined, for the purposes of this paper, to be the kernels $\tn{Ker}\big[\Omega_{R/\ZZ}^n\rightarrow \Omega_{(R/I)/\ZZ}^n\big]$.

The following relationship between the second relative $K$-group and the first group of absolute K\"{a}hler differentials was first pointed out by Bloch \cite{BlochK2Artinian75}.  This result helps establish the base case of the theorem in lemma \hyperref[lembasecase]{\ref{lembasecase}} below.\\

\begin{theorem}\label{theorembloch} Suppose $R$ is a split nilpotent extension of a ring $S$, with extension ideal $I$ 
whose index of nilpotency is $N$.  Suppose further that every positive integer less than or equal to $N$
 is invertible in $S$.  Then 
\begin{equation}K_{2}(R,I)\cong \frac{\Omega_{R,I} ^1}{dI}.\end{equation}
\end{theorem}
\begin{proof}Maazen and Stienstra \cite{MaazenStienstra77}, Example 3.12 page 287.  
\end{proof}

\section{Calculus of Steinberg Symbols and K\"{a}hler Differentials}\label{sectionSteinbergKahler}

\subsection{Notation and Conventions for Symbols and Differentials}\label{subsectionnotation}

The proof in section \hyperref[sectionproof]{\ref{sectionproof}} involves a significant amount of symbolic manipulation.  To streamline this, I use the following notation and conventions:

\begin{enumerate}
\item $R$ is a split nilpotent extension of a $5$-fold stable ring $S$, with extension ideal $I$, 
whose index of nilpotency is $N$.  The multiplicative group of invertible elements of $R$ is denoted by $R^*$.  The subset of elements of the form $1+i$, where $i\in I$, is a subgroup of $R^*$.  It is denoted by $(1+I)^*$ to emphasize its multiplicative structure. 
\item Individual letters, such as $r$ and $r'$, are used to denote elements of $R$.
\item Ordered tuples of elements of $R$ are usually numbered beginning with zero: $(r_0,r_1,...,r_n)$. 
\item ``Bar notation" is often used to abbreviate ordered tuples of elements of $R$.  For example, the expression $(r_0,\bar{r})$ might be used to denote the $(n+1)$-tuple $(r_0,r_1,...,r_n)$, where $\bar{r}$ stands for the last $n$ entries $r_1,...,r_n$.    Similarly, the expression $(\bar{r},r_j,\bar{r}')$ might be used to denote the $n$-tuple $(r_0,...,r_{j-1},r_j,r_{j+1},...,r_n)$, where $\bar{r}$ stands for $r_0,...,r_{j-1}$, and $\bar{r}'$ stands for $r_{j+1},...,r_{n}$.  The number of elements represented by a barred letter is either stated explicitly, determined by context, or immaterial.  In particular, $\bar{r}$ may be empty.  For example, the multiplicativity relation $\{\bar{r},rr',\bar{r}'\}=\{\bar{r},r,\bar{r}'\}\{\bar{r},r',\bar{r}'\}$ in lemma \hyperref[lemMilnorrelations]{\ref{lemMilnorrelations}} below includes the case $\{rr',\bar{r}'\}=\{r,\bar{r}'\}\{r',\bar{r}'\}$, where $\bar{r}$ is empty.
\item Let $(\bar{r})=(r_0,r_1,...,r_n)$ be an ordered $(n+1)$-tuple of elements of $R$.  Then the expression $\bar{r}\in R^*$ means $(r_0,...,r_n)\in (R^*)^{n+1}$.  Similarly, $\{\bar{r}\}$ means the Steinberg symbol corresponding to $\{\bar{r}\}$, if it exists; $d\bar{r}$ means $dr_0\wedge dr_1\wedge...$, and $e^{\bar{r}}$ means $(e^{r_0},e^{r_1},...)$.   
\item Instances of ``capital pi," such as $\Pi$ and $\Pi'$, stand for the products of the entries of tuples such as $(\bar{r})$ and $(\bar{r}')$.(
\item The ``hat notation" $(r_0,...,\hat{r_j},...,r_n)$ denotes the $n$-tuple given by omitting the $j$th entry $r_j$ from the $(n+1)$-tuple $r_0,...,r_j,...,r_n$.   The hat notation may be used to omit multiple entries of an ordered tuple. 
\item The group operation in the Milnor $K$-group $K_n ^{\tn{\fsz{M}}}(R)$ is expressed as multiplication (juxtaposition of Steinberg symbols), although it is actually addition in the Milnor $K$-ring $K^{\tn{\fsz{M}}}(R)$.  
\item The ring multiplication $K_m ^{\tn{\fsz{M}}}(R)\times K_n ^{\tn{\fsz{M}}}(R)\rightarrow K_{m+n} ^{\tn{\fsz{M}}}(R)$ in the Milnor $K$-ring $K^{\tn{\fsz{M}}}(R)$ is expressed abstractly by the symbol $\times$, or concretely by concatenation of the entries of Steinberg symbols.  For example, the distributive law is expressed as 
\[\{\bar{r}\}\times(\{\bar{r}'\}\{\bar{r}''\})=(\{\bar{r}\}\times\{\bar{r}'\})(\{\bar{r}\}\times\{\bar{r}''\})
=\{\bar{r},\bar{r}'\}\{\bar{r},\bar{r}''\},\]
for $\{\bar{r}\}\in K_l ^{\tn{\fsz{M}}}(R), \{\bar{r}'\}\in K_m ^{\tn{\fsz{M}}}(R),$ and $\{\bar{r}''\}\in K_n ^{\tn{\fsz{M}}}(R)$. 
\end{enumerate}


\subsection{Generators and Relations for Milnor $K$-Theory}\label{generatorsMilnor}

In this section, I gather together some elementary results about Steinberg symbols that are useful for the computations in sections \hyperref[subsectionanalysisphi]{\ref{subsectionanalysisphi}} and \hyperref[subsectionanalysispsi]{\ref{subsectionanalysispsi}}.  For numbering consistency, I work with $K_{n+1} ^{\tn{\fsz{M}}}(R)$, rather than $K_{n} ^{\tn{\fsz{M}}}(R)$, since the former group is the one appearing in theorem.  Throughout this section, $n$ is a nonnegative integer.\\

\begin{lem}\label{lemMilnorrelations}As an abstract multiplicative group, $K_{n+1} ^{\tn{\fsz{M}}}(R)$ is generated by the Steinberg symbols $\{r_0,...,r_n\}$, where $r_j\in R^*$ for all $j$, subject to the relations
\begin{enumerate}
\addtocounter{enumi}{-1}
\item $K_{n+1} ^{\tn{\fsz{M}}}(R)$ is abelian.
\item Multiplicative relation: $\{\bar{r},rr',\bar{r}'\}\{\bar{r},r,\bar{r}'\}^{-1}\{\bar{r},r',\bar{r}'\}^{-1}=1$.
\item Steinberg relation: $\{\bar{r},r,1-r,\bar{r}'\}=1$.
\end{enumerate}
\end{lem}
\begin{proof} This follows directly from definition \hyperref[defiMilnorK]{\ref{defiMilnorK}} and the properties of the tensor algebra.  For example, in the case of $K_{2} ^{\tn{\fsz{M}}}(R)$, the tensor product relations in $R^*\otimes_\ZZ R^*$ are 
\[rr'\otimes r''=(r\otimes r'')(r'\otimes r''),\hspace{.3cm} r\otimes r'r''=(r\otimes r')(r\otimes r''),\hspace{.2cm}\tn{and}\hspace{.2cm}r^n\otimes r'=r\otimes(r')^n =(r\otimes r')^n \hspace{.2cm}\tn{for}\hspace{.2cm} n\in\ZZ,\]
since the operations in $R^*$ and $R^*\otimes_\ZZ R^*$ are expressed multiplicatively.  These three relations are equivalent to the multiplicativity relation in the statement of the lemma, while imposing the Steinberg relation is equivalent to quotienting out the Steinberg ideal $I_{\tn{\fsz{St}}}$.  
\end{proof}

Lemma \hyperref[lemMilnorrelations]{\ref{lemMilnorrelations}} translates the na\"{i}ve tensor algebra definition of Milnor $K$-theory into a definition in terms of Steinberg symbols and relations.  The following lemma gathers together additional relations satisfied by Steinberg symbols in the $5$-fold stable case.  The first of these, the idempotent relation, actually requires no stability assumption, but is included here, rather than in lemma \hyperref[lemMilnorrelations]{\ref{lemMilnorrelations}}, because it is information-theoretically redundant.  The other two relations, however, require the ring to have ``enough units." \\

\begin{lem}\label{lemrelationsstable} Let $R$ be a $5$-fold stable ring.  Then the Steinberg symbols $\{r_0,...,r_n\}$ generating $K_{n+1} ^{\tn{\fsz{M}}}(R)$ satisfy the following additional relations:
\begin{enumerate}
\addtocounter{enumi}{2}
\item Idempotent relation: if $e\in R^*$ is idempotent, then $\{\bar{r},e\}=1$ in $K_{n+1}^{\tn{\fsz{M}}}(R)$.
\item Additive inverse relation: $\{\bar{r},r,-r,\bar{r}\}=1$ in $K_{n+1}^{\tn{\fsz{M}}}(R)$.
\item Anticommutativity: $\{\bar{r}, r',r, \bar{r}'\}=\{\bar{r},r,r',\bar{r}'\}^{-1}.$
\end{enumerate}
\end{lem}
\begin{proof}  The idempotent relation requires no stability assumption.  Indeed, multiplicativity implies that $\{\bar{r},e\}=\{\bar{r},e\}\{\bar{r},e\}$.  Multiplying both sides by $\{\bar{r},e\}^{-1}$ yields $\{\bar{r},e\}=1$.  The additive inverse relations and anticommutativity are established in the $5$-fold stable case by Van der Kallen \cite{VanderKallenRingswithManyUnits77} Theorem 8.4, page 509. 
\end{proof}

A few remarks concerning the interdependence of these relations may be helpful.  The additive inverse relation implies anticommutativity, as demonstrated by the following sequence of manipulations, copied from Rosenberg's proof\footnotemark\footnotetext{See Rosenberg \cite{RosenbergK94} Theorem 4.3.15, page 214.  Rosenberg follows Hutchinson's proof \cite{HutchinsonMatsumoto90}, but Hutchinson credits this particular argument to Milnor.} of Matsumoto's theorem: 
\begin{equation}\label{equmatsumotoarg}\begin{array}{lcl} \{r,r'\}&=&\{r,r'\}\{r,-r\}=\{r,-rr'\}\\&=&\{rr'(r')^{-1},-rr'\}=\{rr',-rr'\}\{(r')^{-1},-rr'\}\\&=&\{r',-rr'\}^{-1}=\{r',r\}^{-1}\{r',-r'\}^{-1}\\&=&\{r',r\}^{-1}. \end{array}\end{equation}
However, the additive inverse relation itself depends on expressing the symbol $\{r,-r\}$ as a product of symbols involving the multiplicative and Steinberg relations.  If $R$ is a field, this is easy, since in this case $1-r$ is a unit whenever $r$ is a unit not equal to $1$.   Indeed, from the identity $(1-r)r^{-1}=-(1-r^{-1})$, it follows that $(1-r^{-1})$ is invertible.  Rearranging yields the expression $-r=(1-r)(1-r^{-1})^{-1}.$ Thus,
\begin{equation}\label{equadditiveinversearg}\begin{array}{lcl} \{r,-r\}&=&\{r,(1-r)(1-r^{-1})^{-1}\}=\{r,1-r\}\{r,(1-r^{-1})^{-1}\}\\
&=&\{r,(1-r^{-1})^{-1}\}=\{r^{-1},1-r^{-1}\}\\
&=&1,\end{array}\end{equation}
by repeated application of multiplicatively and the Steinberg relations. If $R$ is a local ring in which $2$ is invertible, then either $1+r$ or $1-r$ is invertible, and again the result is easy.\footnotemark\footnotetext{Suppose neither $1+r$ nor $1-r$ is invertible.  Then both elements belong to the maximal ideal $M$ of $R$, so their difference $2r$ belongs to $M$.  Since $2$ is invertible, $r$ belongs to $M$, a contradiction, since $r$ is invertible. An argument analogous to the one appearing in equation \hyperref[equadditiveinversearg]{\ref{equadditiveinversearg}} now applies.}  Van der Kallen stability is a much more general criterion permitting the same conclusion.  As mentioned above, some authors, such as Elbaz-Vincent and M\"{u}ller-Stach \cite{ElbazVincentMilnor02}, take the additive inverse relation to be part of the {\it definition} of Milnor $K$-theory.  This obviates the need for stability hypotheses in this context, at the expense of the tensor algebra definition \hyperref[defiMilnorK]{\ref{defiMilnorK}}, and consequently, at the expense of the description in terms of differentials given by the main theorem in equation \hyperref[equmaintheorem]{\ref{equmaintheorem}}.

The following lemma is useful for the factorization of Steinberg symbols used in the proof of lemma \hyperref[lempatchingphi]{\ref{lempatchingphi}} below:\\

\begin{lem}\label{lemrelativegenerators} Let $R$ be split nilpotent extension of a ring $S$, with extension ideal $I$. Then the relative Milnor $K$-group $K_{n+1} ^{\tn{\fsz{M}}}(R,I)$ is generated by Steinberg symbols $\{\bar{r}\}=\{r_0,...,r_n\}$ with at least one entry $r_j$ belonging to $(1+I)^*$. 
\end{lem}
\begin{proof}  By the splitting $R=S\oplus I$, any element $r$ of $R^*$ may be written uniquely as a product $s(1+i)$, where $s$ belongs to $S^*$ and $i$ belongs to $I$.  Hence, the Steinberg symbol $\{\bar{r}\}$ may be factored into a product 
\begin{equation}\label{factorizationbarr}\{\bar{r}\}=\{\bar{s}\}\prod_l\{\bar{r}_l'\},\end{equation}
 where each entry of $\{\bar{s}\}$ belongs to $S^*$, and where each factor $\{\bar{r}_l'\}$ has at least one entry in $(1+I)^*$.   For example, the Steinberg symbol $\{r_0,r_1\}$ factors as follows:
\[\{r_0,r_1\}=\{s_0,s_1\}\{s_0,1+i_1\}\{1+i_0,s_1\}\{1+i_0,1+i_1\}.\]
In general, there are $2^{n+1}-1$ factors of the form $\{\bar{r}_l'\}$ in equation \hyperref[factorizationbarr]{\ref{factorizationbarr}}.  Under the map $K_{n+1} ^{\tn{\fsz{M}}}(R)\rightarrow K_{n+1} ^{\tn{\fsz{M}}}(S)$ induced by the canonical surjection $R\rightarrow S$, the factors $\{\bar{r}_l'\}$ all map to the identity by the idempotent relation. Therefore, $\{\bar{r}\}$ maps to $\{\bar{s}\}$.   Hence, $\{\bar{r}\}$ belongs to the kernel $\tn{Ker}[K_{n+1} ^{\tn{\fsz{M}}}(R)\rightarrow K_{n+1} ^{\tn{\fsz{M}}}(S)]=K_{n+1} ^{\tn{\fsz{M}}}(R,I)$ if and only if $\{\bar{s}\}=1$.   Thus, each element $\{\bar{r}\}$ of $K_{n+1} ^{\tn{\fsz{M}}}(R,I)$ admits a factorization $\prod_l\{\bar{r}_l'\}$, where each factor $\{\bar{r}_l'\}$ has at least one entry in $(1+I)^*$.  
\end{proof}



\subsection{Generators and Relations for K\"{a}hler Differentials}\label{subsectiongeneratorsKahler}

In this section, I describe the groups $\Omega_{R}^{n}/d\Omega_{R}^{n-1}$ and $\Omega_{R,I}^{n}/d\Omega_{R,I}^{n-1}$ in more detail in terms of generators and relations.  

\begin{lem}\label{lemKahlerrelations}As an abstract additive group, $\Omega_{R/\ZZ}^{n}/d\Omega_{R/\ZZ}^{n-1}$ is generated by  differentials $r_0dr_1\wedge...\wedge dr_n$, where $r_j\in R$ for all $j$, subject to the relations
\begin{enumerate}
\addtocounter{enumi}{-1}
\item $\Omega_{R/\ZZ}^{n+1}/d\Omega_{R/\ZZ}^{n}$ is abelian.
\item Additive relation: $(r+r')d\bar{r}=rd\bar{r}+r'd\bar{r}$.  
\item Leibniz rule: $rd(r'r'')\wedge d\bar{r}=rr'dr''\wedge d\bar{r}+rr''dr'\wedge d\bar{r}$.  
\item Alternating relation: $r_0dr_1\wedge...\wedge dr_j\wedge dr_{j+1}\wedge...\wedge dr_n=-r_0dr_1\wedge...\wedge dr_{j+1}\wedge dr_j\wedge...\wedge dr_n$.
\item Exactness: $d\bar{r}=0$. 
\end{enumerate}
\end{lem}
\begin{proof} This follows directly from definition \hyperref[defikahler]{\ref{defikahler}} and the properties of the exterior algebra.   
\end{proof}
These relations, of course, imply other familiar relations.  For example, the Leibniz rule and exactness together imply that
\[r_1dr_0\wedge dr_2\wedge...\wedge dr_n= -r_0dr_1\wedge dr_2\wedge...\wedge dr_n,\]
 so the alternating property ``extends to coefficients."   This, in turn, implies that additivity is not ``confined to coefficients:" 
\[r_0dr_1\wedge...\wedge d(r_j+r_j')\wedge...\wedge dr_n=r_0dr_1\wedge...\wedge dr_j\wedge...\wedge dr_n+r_0dr_1\wedge...\wedge dr_j'\wedge...\wedge dr_n.\]
Similarly, repeated use of the alternating relation implies that applying a permutation to the elements $r_0,...,r_n$ appearing in the differential $r_0dr_1\wedge...\wedge dr_n$ yields the same differential, multiplied by the sign of the permutation. 

The following lemma establishes properties of K\"{a}hler differentials analogous to the properties of Milnor $K$-groups established in lemma \hyperref[lemrelativegenerators]{\ref{lemrelativegenerators}}.  Minor stability and invertibility assumptions are necessary to yield the desired results.\\

\begin{lem}\label{lemrelativekahlergenerators} Let $R$ be split nilpotent extension of a $2$-fold stable ring $S$, in which $2$ is invertible.  Let $I$ be the extension ideal.  
\begin{enumerate}
\item The group $\Omega_{R,I}^n$ of absolute K\"{a}hler differentials of degree $n$ relative to $I$ is generated by differentials of the form $rd\bar{r}\wedge d\bar{r}'$, where $r$ is either $1$ or belongs to $I$, where $\bar{r}\in I$, and where $\bar{r}'\in R^*$.  
\item The group $\Omega_{R,I}^{n}/d\Omega_{R,I}^{n-1}$ is a subgroup of the group $\Omega_{R/\ZZ}^{n}/d\Omega_{R/\ZZ}^{n-1}$.  In particular, $d\Omega_{R,I}^{n-1}=d\Omega_{R}^{n-1}\cap \Omega_{R,I}^{n}$.  
\end{enumerate}
\end{lem}
\begin{proof}  For the first part of the lemma, note that the elements $r_0,...,r_n$ contributing to a differential $r_0dr_1\wedge...\wedge dr_n$ may be permuted up to sign, using exactness and the alternating property of the exterior product.  Also note that by lemma \hyperref[lemstability]{\ref{lemstability}}, any element of $R$ may be written as a sum of two units.  It therefore suffices to show that $\Omega_{R,I}^n$ is generated by differentials of the form $r'd\bar{r}'$, where either $r'$ or at least one entry of $\bar{r}'$ belongs to $I$.  By the splitting $R=S\oplus I$, any element $r$ of $R^*$ may be written uniquely as a sum $s+i$, where $s$ belongs to $S$ and $i$ belongs to $I$.  Hence, the differential $rd\bar{r}$ may be decomposed into a sum 
\begin{equation}\label{decompositionrdbarr}rd\bar{r}=sd\bar{s}+\sum_lr'_ld\bar{r}_l',\end{equation}
where $s$ and each entry of $\bar{s}$ belong to $S$, and where for each $l$, either $r'_l$ or at least one entry of $d\bar{r}_l'$ belongs to $I$.  For example, the differential $r_0dr_1\wedge dr_2$ decomposes as follows:
\begin{equation*}\begin{array}{lcl} r_0dr_1\wedge dr_2&=&s_0ds_1\wedge ds_2+s_0ds_1\wedge di_2+s_0di_1\wedge ds_2+s_0di_1\wedge di_2\\&+&i_0ds_1\wedge ds_2+i_0ds_1\wedge di_2+i_0di_1\wedge ds_2+i_0di_1\wedge di_2.\end{array}\end{equation*}
In general, there are $2^{n+1}-1$ factors of the form $d\bar{r}_l'$ in equation \hyperref[decompositionrdbarr]{\ref{decompositionrdbarr}}.   Under the map $\Omega_{R}^n\rightarrow\Omega_{S}^n$ induced by the canonical surjection $R\rightarrow S$, the summands $r'_ld\bar{r}_l'$ all map to zero, so $rd\bar{r}$ maps to $sd\bar{s}$.   Hence, $rd\bar{r}$ belongs to the kernel $\tn{Ker}[\Omega_{R}^n\rightarrow\Omega_{S}^n]=\Omega_{R,I}^n$ if and only if $sd\bar{s}=0$.   Thus, each element $rd\bar{r}$ of $\Omega_{R,I}^n$ admits a decomposition $rd\bar{r}=\sum_lr'_ld\bar{r}_l'$, where or each $l$, either $r'_l$ or at least one entry of $d\bar{r}_l'$ belongs to $I$.  

For the second part of the lemma, the inclusion $d\Omega_{R,I}^{n-1}\subset d\Omega_{R}^{n-1}\cap \Omega_{R,I}^{n}$ is obvious, independent of the fact that $R$ is a split nilpotent extension of $S$.  Conversely, suppose that $\omega$ belongs to the intersection $d\Omega_{R}^{n-1}\cap \Omega_{R,I}^{n}$.    Since $\omega\in \Omega_{R,I}^{n}$, the first part of the lemma implies that $\omega$ may be expressed as a sum of differentials of the form $rd\bar{r}=r_0dr_1\wedge...\wedge dr_n$, where either $r_0$ or at least one of the $r_j$ belongs to $I$.  Since $\omega\in d\Omega_{R}^{n-1}$, the ``coefficient" $r_0$ in each summand may be taken to be $1$.  Hence, $\omega$ is a sum of terms of the form $dr_1\wedge...\wedge dr_n=d\big(r_1dr_2\wedge...\wedge dr_n\big)$, where one of the elements $r_1,...,r_n$ belongs to $I$, so $\omega\in d\Omega_{R,I}^{n-1}$ by the first part of the lemma.   It follows that the map sending $\omega+d\Omega_{R,I}^{n-1}$ to $\omega+d\Omega_{R}^{n-1}$ is an injective group homomorphism from $\Omega_{R,I}^{n}/d\Omega_{R,I}^{n-1}\rightarrow\Omega_{R}^{n}/d\Omega_{R}^{n-1}$, which may be regarded as an inclusion map.  

\end{proof}


\subsection{The $d\log$ Map; the de Rham-Witt Viewpoint}\label{subsectiondlogdeRhamWitt}

The following lemma establishes the existence of the ``canonical $d\log$ map" from Milnor $K$-theory to the absolute K\"{a}hler differentials, used in the proof of lemma \hyperref[lempatchingphi]{\ref{lempatchingphi}} below.\\

\begin{lem}\label{lemdlog} Let $R$ be a commutative ring.  The map $R^*\rightarrow\Omega_{R/\ZZ}^1$ sending $r$ to $d\log(r)=dr/r$ extends to a homomorphism $d\log:T_{R^*/\mathbb{Z}}\rightarrow \Omega_{R/\ZZ}^\bullet$ of graded rings, by sending sums to sums and tensor products to exterior products.   This homomorphism induces a homomorphism of graded rings:
\[d\log:K_\bullet^{\tn{\fsz{M}}}(R)\longrightarrow \Omega_{R/\ZZ}^\bullet\]
\begin{equation}\label{equdlog}\{r_0,...,r_n\}\mapsto\displaystyle\frac{dr_0}{r_0}\wedge...\wedge\frac{dr_n}{r_n}.\end{equation}
\end{lem}
\begin{proof} The map $d\log:T_{R^*/\mathbb{Z}}\rightarrow \Omega_{R/\ZZ}^*$ is a graded ring homomorphism by construction, since its definition stipulates that sums are sent to sums and tensor products to exterior products.  Elements of the form $r\otimes(1-r)$ in $R*\otimes_\ZZ R^*$ map to zero in $\Omega_{R/\ZZ}^2$ by the alternating property of the exterior product:
\[d\log\big(r\otimes(1-r)\big)=\frac{dr}{r}\wedge\frac{d(1-r)}{1-r}=-\frac{1}{r(1-r)}dr\wedge dr=0,\]
so $d\log$ descends to a homomorphism $K_\bullet^{\tn{\fsz{M}}}(R)\longrightarrow \Omega_{R/\ZZ}^\bullet$. 
\end{proof}

Lars Hesselholt \cite{HesselholtBigdeRhamWitt} provides a more sophisticated viewpoint regarding the $d\log$ map and the closely related map $\phi_{n+1}$ in the main theorem, expressed in equation \hyperref[equationphi]{\ref{equationphi}} above.  This viewpoint is expressed in terms of pro-abelian groups, de Rham Witt complexes, and Frobenius endomorphisms.   In describing it, I closely paraphrase a private communication \cite{Hesselholtprivate14} from Hesselholt.  For every commutative ring $R$, there exists a map of pro-abelian groups
\[d\log: K_n^{\tn{\fsz{M}}}(R)\rightarrow W\Omega_{R/\ZZ}^n\]
from Milnor $K$-theory to an appropriate de Rham-Witt theory $W\Omega_{R/\ZZ}^n$, taking the Steinberg symbol $\{r_1,...,r_n\}$ to the element $d\log[r_1]...d\log[r_n]$, where $[r]$ is the Teichm\"{u}ller representative of $r$ in an appropriate ring of Witt vectors of $R$.   Here, $W\Omega_{R/\ZZ}^n$ may represent either the $p$-typical de Rham-Witt groups or the big Rham-Witt groups.  On the $p$-typical de Rham-Witt complex, there is a (divided) Frobenius endomorphism $F=F_p$; on the big de Rham-Witt complex, there is a (divided) Frobenius endomorphism $F_n$ for every positive integer $n$.  The map $d\log$ maps into the sub-pro-abelian group 
\[(W\Omega_{R/\ZZ}^n)^{F=\tn{Id}}\subset W\Omega_{R/\ZZ}^n,\]
fixed by the appropriate Frobenius endomorphism or endomorphisms.  Using the big de Rham complex, one may conjecture\footnotemark\footnotetext{This is Hesselholt's idea.} that for every commutative ring $R$ and every nilpotent ideal $I\subset R$, the induced map of relative groups
\[K_n^{\tn{\fsz{M}}}(R,I)\rightarrow(W\Omega_{R,I}^n)^{F=\tn{Id}},\]
is an isomorphism of pro-abelian groups.  Expressing the right-hand-side in terms of differentials, as in the main theorem \hyperref[equmaintheorem]{\ref{equmaintheorem}}, likely requires some additional hypotheses.\footnotemark\footnotetext{Hesselholt  \cite{Hesselholtprivate14} writes, {\it ``If every prime number $l$ different from a fixed prime number $p$ is invertible in $R$, then one should be able to use the $p$-typical de Rham-Witt groups instead of the big de Rham-Witt groups. In this context, [the main theorem] can be seen as a calculation of this Frobenius fixed set... ... In order to be able to express the Frobenius fixed set in terms of differentials (as opposed to de Rham-Witt differentials), I would think that it is necessary to invert $N$... ...I do not think that inverting 2 is enough.}}


\section{Proof of the Theorem}\label{sectionproof}

\subsection{Strategy of Proof}\label{subsectionstrategy}

The proof of the theorem is by induction on $n$ in the statement $\displaystyle K_{n+1} ^{\tn{\fsz{M}}}(R,I)\cong \Omega_{R,I} ^n/d\Omega_{R,I} ^{n-1}$.   The base case of the theorem ($n=1$) is provided by combining the theorems \hyperref[theoremvanderkallen]{\ref{theoremvanderkallen}} (Van der Kallen), \hyperref[theoremmaazen]{\ref{theoremmaazen}} (Maazen and Stienstra), and  \hyperref[theorembloch]{\ref{theorembloch}} (Bloch), introduced in section \hyperref[sectionprelim]{\ref{sectionprelim}} above.   This synthesis is discussed in detail in section \hyperref[subsectionbasecase]{\ref{subsectionbasecase}} below.  The induction hypothesis assumes the existence of isomorphisms  $\phi_{m+1}:\displaystyle K_{m+1} ^{\tn{\fsz{M}}}(R,I)\rightarrow\Omega_{R,I} ^m/d\Omega_{R,I} ^{m-1}$ and $\psi_{m+1}:\displaystyle \Omega_{R,I} ^m/d\Omega_{R,I} ^{m-1}\rightarrow K_{m+1} ^{\tn{\fsz{M}}}(R,I)$ for $1\le m<n$, satisfying conditions made explicit in section \hyperref[subsectionexplicitinduction]{\ref{subsectionexplicitinduction}}.   These isomorphisms are then used to construct corresponding isomorphisms $\phi_{n+1}$ and $\psi_{n+1}$ in sections \hyperref[subsectionanalysisphi]{\ref{subsectionanalysisphi}} and \hyperref[subsectionanalysispsi]{\ref{subsectionanalysispsi}}.  

As illustrated by equations \hyperref[equationphi]{\ref{equationphi}} and \hyperref[equationpsi]{\ref{equationpsi}} in section \hyperref[sectionintroduction]{\ref{sectionintroduction}} above, it is easy to specify the images of certain special generators of $K_{n+1} ^{\tn{\fsz{M}}}(R,I)$ and $\Omega_{R,I} ^n/d\Omega_{R,I} ^{n-1}$ under $\phi_{n+1}$ and $\psi_{n+1}$.  The whole ``difficulty" of the proof is in verifying that these formulae actually extend to well-defined isomorphisms.  Induction allows this problem to be split into two parts: first, to show that the maps $\phi_{m+1}$ and $\psi_{m+1}$ for $1\le m<n$ are well-defined isomorphisms; second, that these maps give rise to well-defined isomorphisms $\phi_{n+1}$ and $\psi_{n+1}$.  The first part ``comes for free," via the base case of the theorem and the induction hypothesis.  The second part consists of ``patching together" $(n+1)$ maps $\Phi_{n+1,j}$ and $\Psi_{n+1,j}$, defined, roughly speaking, by applying $\phi_n$ and $\psi_n$ to Steinberg symbols and  differentials ``of size $n$,"  given by omitting individual entries of corresponding symbols and differentials ``of size $n+1$."  The maps $\Phi_{n+1,j}$ and $\Psi_{n+1,j}$ are introduced in definitions \hyperref[defiPhinj]{\ref{defiPhinj}} and \hyperref[defiPsinplusonej]{\ref{defiPsinplusonej}}, respectively.  The ``patching lemmas" \hyperref[lempatchingphi]{\ref{lempatchingphi}} and \hyperref[lempatchingpsi]{\ref{lempatchingpsi}} are the most computationally involved parts of the proof. 

The obvious question raised by this approach is, ``why not just define the images of sets of generators of  $\displaystyle K_{n+1} ^{\tn{\fsz{M}}}(R,I)$ and $\Omega_{R,I} ^n/d\Omega_{R,I} ^{n-1}$, respectively, show that they satisfy the proper relations in the target, etc, instead of pursuing an elaborate induction and patching scheme?"  There may well be some clever way to do this, and to convince oneself that all the necessary conditions have been checked, but it is not a straightforward procedure.   The reason why is that relations among ``convenient" generators for $\displaystyle K_{n+1} ^{\tn{\fsz{M}}}(R,I)$ generally do not translate easily into relations among ``convenient" generators for $\Omega_{R,I} ^n/d\Omega_{R,I} ^{n-1}$, and vice versa.  For example, consider the Leibniz rule in lemma \hyperref[subsectiongeneratorsKahler]{\ref{subsectiongeneratorsKahler}}:
\[rd(r'r'')\wedge d\bar{r}=rr'dr''\wedge d\bar{r}+rr''dr'\wedge d\bar{r},\]
and suppose that we want to verify that 
\begin{equation}\label{equleibnizcomplication}\psi_{n+1}\big(rd(r'r'')\wedge d\bar{r}\big)=\psi_{n+1}\big(rr'dr''\wedge d\bar{r}\big)\psi_{n+1}\big(rr''dr'\wedge d\bar{r}\big).\end{equation}
Assume for simplicity that the element $r$ and the product $r'r''$ both belong to $I$, and that every entry of the $(n-1)$-tuple $(\bar{r})$ is an element of $R^*$; the example will raise sufficient subtleties for illustrative purposes even in this special case.  Equation  \hyperref[equationpsi]{\ref{equationpsi}} specifies the image in $\displaystyle K_{n+1} ^{\tn{\fsz{M}}}(R,I)$ of the left-hand side of equation \hyperref[equleibnizcomplication]{\ref{equleibnizcomplication}}:
\[\psi_{n+1}\big(rd(r'r'')\wedge d\bar{r}\big)=\{e^{r\Pi},e^{r'r''},\bar{r}\},\]
where $\Pi$ is the product of the entries of $\bar{r}$.   However, it is not straightforward to write out the right-hand side of equation \hyperref[equleibnizcomplication]{\ref{equleibnizcomplication}} explicitly.   In particular, the factors $r'$ and $r''$ of the product $r'r''$ may both belong to $I$, or only one may belong to $I$, or neither may belong to $I$, if $I$ is not prime.  Suppose for simplicity that $r'$ belongs to $I$ but $r''$ does not.  Then the second factor on the right-hand side of equation \hyperref[equleibnizcomplication]{\ref{equleibnizcomplication}} is 
\[\psi_{n+1}\big(rr''dr'\wedge d\bar{r}\big)=\{e^{rr''\Pi},e^{r'},\bar{r}\},\]
but the explicit form of the first factor cannot be read off from equation \hyperref[equationpsi]{\ref{equationpsi}}, since $r''$ is neither a unit nor an element of $I$.  Instead, one must use lemma \hyperref[lemstability]{\ref{lemstability}} to write $r''$ as a sum of two units $r''=u+v$.  Then the right-hand side of equation \hyperref[equleibnizcomplication]{\ref{equleibnizcomplication}} may be written out explicitly:
\[\psi_{n+1}\big(rd(r'r'')\wedge d\bar{r}\big)=\{e^{rr'u\Pi},u,\bar{r}\}\{e^{rr'v\Pi},v,\bar{r}\}\{e^{rr''\Pi},e^{r'},\bar{r}\}.\]
However, one still must show that the right-hand side does not depend on the choice of $u$ and $v$, since $r''$ can generally be written as a sum of two units in many different ways.\footnotemark\footnotetext{See lemma \hyperref[lemuvUV]{\ref{lemuvUV}} below.}  This example should serve to convince the reader that a na\"{i}ve, straightforward approach to the proof involves many cases and loose ends.  The induction approach I use instead has the advantage of being more systematic.  Most of the work involved in checking relations is shunted off on the induction hypothesis, with the tradeoff that one must endure a bit of computation to show that the maps $\Phi_{n+1,j}$ and $\Psi_{n+1,j}$ really patch together as claimed.


\subsection{Base Case of the Theorem}\label{subsectionbasecase}

Combining several of the preliminary results in section \hyperref[sectionprelim]{\ref{sectionprelim}} yields the following lemma, which serves as the base case of the theorem:\\

\begin{lem}\label{lembasecase} Suppose $R$ is a split nilpotent extension of a $5$-fold stable ring $S$, with extension ideal $I$ 
whose index of nilpotency is $N$.  Suppose further that every positive integer less than or equal to $N$
 is invertible in $S$.  Then
\begin{equation}K_{2} ^{\tn{\fsz{M}}}(R,I)\cong \frac{\Omega_{R,I} ^1}{dI}\end{equation}
\end{lem}
\begin{proof}$R$ is $5$-fold stable by lemma \hyperref[lemstability]{\ref{lemstability}}, so $K_2 ^{\tn{\fsz{M}}}(R)\cong D_2(R)$ by theorem \hyperref[theoremvanderkallen]{\ref{theoremvanderkallen}}. $R$ is a split radical extension of $S$ since $I$ is nilpotent,  so $K_2 ^{\tn{\fsz{M}}}(R,I)\cong D_2(R,I)\cong K_2(R,I)$ by theorem \hyperref[theoremmaazen]{\ref{theoremmaazen}}.  Finally, since  every positive integer less than or equal to $N$ is invertible in $S$, $K_2(R,I)\cong \Omega_{R,I} ^1/dI$ by theorem \hyperref[theorembloch]{\ref{theorembloch}}.
\end{proof}

In terms of Steinberg symbols and K\"{a}hler differentials, the isomorphisms of lemma \hyperref[lembasecase]{\ref{lembasecase}} are the maps 

\[\phi_{2}:K_{2} ^{\tn{\fsz{M}}}(R,I)\longrightarrow\frac{\Omega_{R,I} ^1}{dI}\]
\begin{equation}\label{equationphi1}
\{r_0,r_1\}\mapsto\log(r_0)\frac{dr_1}{r_1},
\end{equation}
where $r_0\in (1+I)^*$ and $r_1\in R^*$, and 
\[\psi_{2}:\frac{\Omega_{R,I} ^1}{dI}\longrightarrow K_{2} ^{\tn{\fsz{M}}}(R,I)\]
\begin{equation}\label{equationpsi1}
r_0dr_1\mapsto\{e^{r_0r_1},r_{1}\},
\end{equation}
where $r_0\in I$ and $r_{1}\in R^*$.  These maps are given by setting $n=1$ in equations \hyperref[equationphi]{\ref{equationphi}} and \hyperref[equationpsi]{\ref{equationpsi}} of section \hyperref[sectionintroduction]{\ref{sectionintroduction}}.  I will now describe in more detail how they arise.  As described in the proof of lemma \hyperref[lembasecase]{\ref{lembasecase}}, $\phi_2$ may be viewed as a composition of isomorphisms
\[K_{2} ^{\tn{\fsz{M}}}(R,I)\rightarrow D_2(R,I)\rightarrow  \frac{\Omega_{R,I} ^1}{dI}.\]
The first isomorphism is given, in the $5$-fold stable case, by restricting the isomorphism $K_{2} ^{\tn{\fsz{M}}}(R)\cong D_2(R)$ of lemma \hyperref[subsectionsymbolicstability]{\ref{subsectionsymbolicstability}} to the subgroup $K_{2} ^{\tn{\fsz{M}}}(R,I)$. This isomorphism is described explicitly by Van der Kallen \cite{VanderKallenRingswithManyUnits77} theorem 8.4, page 509, as the map taking the Steinberg symbol $\{r_0,r_1\}$ to the Dennis-Stein symbol $\langle(r_0-1)/r_1,r_1\rangle$.   Restricting to $K_{2} ^{\tn{\fsz{M}}}(R,I)$, one may assume that at least one of the entries $r_0$ and $r_1$ of $\{r_0,r_1\}$ belongs to $(1+I)^*$.  By anticommutativity, one may assume that $r_0\in(1+I)^*$.  The second isomorphism is described explicitly by Maazen and Stienstra \cite{MaazenStienstra77} section 3.12, pages 287-289, as the map taking the Dennis-Stein symbol $\langle a,b\rangle$ to the differential $\log(1+ab)(db/b)$.\footnotemark\footnotetext{Actually, Maazen and Stienstra give the image as $\log(1+ab)(da/a)$, but the Dennis-Stein relation $\langle a,b\rangle\langle -b,-a\rangle=1$ in definition \hyperref[defidennisstein]{\ref{defidennisstein}} implies that the definition I use here is equivalent.} This definition makes sense whether or not $b$ is invertible, since every term in the power series expansion of $\log(1+ab)$ is divisible by $b$.  Putting the two maps together, 
\[\phi_2\big(\{r_0,r_1\}\big)= \log\Big(1+\frac{r_0-1}{r_1}r_1\Big)\frac{dr_1}{r_1}=\log(1+r_0)\frac{dr_1}{r_1},\]
as stated in equation \hyperref[equationphi1]{\ref{equationphi1}}.  

Similarly, $\psi_2$ may be viewed as a composition of isomorphisms in the opposite direction: 
\[\frac{\Omega_{R,I} ^1}{dI}\rightarrow D_2(R,I)\rightarrow  K_{2} ^{\tn{\fsz{M}}}(R,I).\]
The first isomorphism is described explicitly by Maazen and Stienstra \cite{MaazenStienstra77} section 3.12, pages 287-289, as the map taking the differential $r_0dr_1$ to the Dennis-Stein symbol $\langle (e^{r_0r_1}-1)/r_1,r_1\rangle$.\footnotemark\footnotetext{Maazen and Stienstra give the image as $\langle (e^{r_0r_1}-1)/r_0,r_0\rangle$, but the Leibniz rule $d(r_0r_1)=r_0dr_1+r_1dr_0$, together with exactness, proves that the definition I use here is equivalent.}  Restricting to $D_2(R,I)$, one may assume that at least one of the elements $r_1$ and $r_2$ belongs to $I$.   By exactness and the Leibniz rule, it suffices to describe the images of differentials of the form $idr$.  Such a differential maps to $\langle (e^{ir}-1)/r,r\rangle$.  The second isomorphism is given, in the $5$-fold stable case, by restricting the isomorphism $D_2(R)\rightarrow K_{2} ^{\tn{\fsz{M}}}(R)$ to the subgroup $D_2(R,I)$.   This isomorphism is described explicitly by Van der Kallen \cite{VanderKallenRingswithManyUnits77} theorem 8.4, page 509, as the map taking the Dennis-Stein symbol $\langle a,b\rangle$ to the Steinberg symbol $\{1+ab,b\}$.  Putting the two maps together, 
\[\psi_2(r_0dr_1)=\Big\{1+\frac{e^{r_0r_1}-1}{r_1}r_1,r_1\Big\} =\{e^{r_0r_1},r_1\},\]
as stated in equation \hyperref[equationpsi1]{\ref{equationpsi1}}.


\subsection{Induction Hypothesis}\label{subsectionexplicitinduction}

Let $n$ be a positive integer.  The induction hypothesis states that for any positive integer $m$ less than $n$, there exists an isomorphism
\[\phi_{m+1}:K_{m+1} ^{\tn{\fsz{M}}}(R,I)\rightarrow \frac{\Omega_{R,I} ^m}{d\Omega_{R,I} ^{m-1}}\]
\begin{equation}\label{equinductionphi}\{r,\bar{r}\}\mapsto\log(r)\frac{d\bar{r}}{\Pi},\end{equation}
for $r\in (1+I)^*$ and $\bar{r}\in R^*$, with inverse
\[\psi_{m+1}:\frac{\Omega_{R,I} ^m}{d\Omega_{R,I} ^{m-1}}\rightarrow K_{m+1} ^{\tn{\fsz{M}}}(R,I)\]
\begin{equation}\label{equinductionpsi}rd\bar{r}\wedge d\bar{r}'\mapsto\{e^{r\Pi'},e^{\bar{r}},\bar{r}'\},\end{equation}
for $r,\bar{r}\in I$ and $\bar{r}'\in R^*$.   Remaining is the inductive step: to show that the induction hypothesis implies the existence of such isomorphisms for $m=n$.

\subsection{Definition and Analysis of the Map $\displaystyle\phi_{n+1}: K_{n+1} ^{\tn{\fsz{M}}}(R,I)\rightarrow \Omega_{R,I} ^n/d\Omega_{R,I} ^{n-1}$}\label{subsectionanalysisphi}

I will define $\phi_{n+1}$ in several steps, building up its properties in the process.\\
\begin{defi}\label{defprelimmapsphi} Let $m$ be a positive integer.
\begin{enumerate}
\item Let $F_{m+1}(R^*)$ be the free abelian group generated by ordered $(m+1)$-tuples $(r_0,...,r_m)$ of elements of $R^*$.  
\item Let $q_{m+1}$ be the quotient homomorphism
\[q_{m+1}: F_{m+1}(R^*)\rightarrow K_{m+1} ^{\tn{\fsz{M}}}(R)\]
\[(r_0,...,r_m)\mapsto\{r_0,...,r_m\},\]
sending each $(m+1)$-tuple to the corresponding Steinberg symbol.  
\item Let $F_{m+1}\big(R^*,(1+I)^*\big)$ be the preimage in $F_{m+1}(R^*)$ of the relative Milnor $K$-group $K_{m+1}^{\tn{\fsz{M}}}(R,I)\subset K_{m+1}^{\tn{\fsz{M}}}(R)$ under $q_{m+1}$.  
\item For $1\le m<n$, let $\Phi_{m+1}$ be the composition 
\[\phi_{m+1}\circ q_{m+1}:F_{m+1}\big(R^*,(1+I)^*\big)\rightarrow\frac{\Omega_{R,I} ^{m}}{d\Omega_{R,I} ^{m-1}}.\]
\end{enumerate}
\end{defi}

An $(m+1)$-tuple $(r_0,...,r_m)$ of elements of $R^*$ satisfying the condition that at least one of its entries belongs to $(1+I)^*$ is automatically an element of $F_{m+1}\big(R^*,(1+I)^*\big)$.  If $1\le m<n$, then specifying such an element $r_j$ allows the image of $(r_0,...,r_m)$ under $\Phi_m$ to be expressed explicitly in terms of logarithms and their differentials, via equation \hyperref[equinductionphi]{\ref{equinductionphi}} above. In particular, if $r_0\in(1+I)^*$, then
\[\Phi_{m+1}(r_0,...,r_m)=\log(r_0)\frac{dr_1}{r_1}\wedge...\wedge\frac{dr_{m}}{r_{m}}.\]
However, there are generally $(m+1)$-tuples belonging to $F_{m+1}\big(R^*,(1+I)^*\big)$ that do not satisfy this property.  For example, any $(m+1)$-tuple including an idempotent element as one of its entries maps to the trivial element of $K_{m+1}^{\tn{\fsz{M}}}(R)$ under $q_{m+1}$, and therefore belongs to $F_{m+1}\big(R^*,(1+I)^*\big)$, whether or not it includes an element of $(1+I)^*$.  

It will be useful to consider diagrams of the form 
\[\begin{CD}
K_{n+1} ^{\tn{\fsz{M}}}(R,I) @<q_{n+1}<< F_{n+1}\big(R^*,(1+I)^*\big)
@<A_j<<F_{n}\big(R^*,(1+I)^*\big)\times R^*
@>\beta>>\displaystyle\frac{\Omega_{R,I} ^{n-1}}{d\Omega_{R,I} ^{n-2}}\times R^*
@>{\gamma}>>\displaystyle\frac{\Omega_{R,I}^{n}}{d\Omega_{R,I}^{n-1}},
\end{CD}\]
where the maps $A_j$, $\beta$, and $\gamma$ are defined as follows:\\
\begin{defi}\label{defiAjbetagamma} Let $j$ be a nonnegative integer less than or equal to $n$. 
\begin{enumerate}
\item Let $A_j$ be the map converting each $n$-tuple $(\bar{r})$ in $F_{n}\big(R^*,(1+I)^*\big)$ to an $(n+1)$-tuple in $F_{n+1}\big(R^*,(1+I)^*\big)$ by inserting an element of $R^*$ between the $j-1$ and $j$th entries of $(\bar{r})$.  Extend $A_j$ to inverses and products of $n$-tuples to produce products of $(n+1)$-tuples and their inverses sharing the same $j$th entries.  For example, if $n=3$ and $j=1$, then 
\[A_1 \big((r_0,r_1,r_2)(r_0',r_1',r_2')^{-1}, r\big)=(r_0,r,r_1,r_2)(r_0',r,r_1',r_2')^{-1}.\] 
\item Let $\beta$ be the Cartesian product $\Phi_{n}\times \tn{Id}_{R^*}$, where $\tn{Id}_{R^*}$ is the identity map on $R^*$.
\item Let $\gamma$ be defined by wedging on the right with $\displaystyle\frac{dr}{r}$ for $r\in R^*$, sending $\big(\omega+d\Omega_{R,I} ^{n-2},r\big)$ to $\displaystyle\omega\wedge\frac{dr}{r}+d\Omega_{R,I} ^{n-1}$.
\end{enumerate}
\end{defi}
It is prudent to verify that these maps are well-defined.  For $q_{n+1}$ and $\beta$, this is obvious.  For $A_j$, the definition certainly produces an element of $F_{n+1}(R^*)$, and it remains to show that this element belongs to the subgroup $F_{n+1}\big(R^*,(1+I)^*\big)$.  To see this, note that inserting an element corresponds, after applying $q_{n+1}$ and anticommutativity, to the product $K_n^{\tn{\fsz{M}}}(R)\times K_1^{\tn{\fsz{M}}}(R)\rightarrow K_{n+1}^{\tn{\fsz{M}}}(R)$, which fits into the commutative square
\[\begin{CD}
K_n^{\tn{\fsz{M}}}(R)\times K_1^{\tn{\fsz{M}}}(R)@>>>K_{n+1}^{\tn{\fsz{M}}}(R)\\
@VVV@VVV\\
K_n^{\tn{\fsz{M}}}(S)\times K_1^{\tn{\fsz{M}}}(S)@>>>K_{n+1}^{\tn{\fsz{M}}}(S).
\end{CD}\]
This implies that a generator of $F_{n+1}(R^*)$ given by inserting a element of $R^*$ between the $j-1$ and $j$th entries of a generator of $F_n(R^*)$ maps to the identity in $K_{n+1}^{\tn{\fsz{M}}}(S)$ under $q_{n+1}$ if the original $n$-tuple maps to the identity in $K_{n}^{\tn{\fsz{M}}}(S)$ under $q_{n}$.\footnotemark\footnotetext{The converse is obviously false; for example, inserting $1$ into any generator of $F_n(R^*)$ produces an element of $F_{n+1}\big(R^*,(1+I)^*\big)$ by the idempotent lemma.}

Turning to $\gamma$, it is necessary to verify that the image $\omega\wedge(dr/r)+d\Omega_{R,I} ^{n-1}$ of the pair $\big(\omega+d\Omega_{R,I} ^{n-2},r\big)$ does not depend on the choice of $\omega$; i.e., that adding an exact differential to $\omega$ does not alter the image.  By lemma \hyperref[lemrelativekahlergenerators]{\ref{lemrelativekahlergenerators}}, this reduces to showing that $d\bar{r}\wedge(dr/r)$ belongs to $d\Omega_{R,I} ^{n-1}$ for any exact differential of the form $d\bar{r}=dr_0\wedge...\wedge dr_{n-2}$ with at least one $r_j$ belonging to $I$. But $d\bar{r}\wedge(dr/r)$ is just $d\big(\log(r)d\bar{r}\big)$ up to sign, and at least one of its factors is the differential of an element of $I$, by the choice of $\bar{r}$.   

A few other properties of the maps $A_j$  are noteworthy.  First, they respect the group structure of $F_n \big(R^*,(1+I)^*\big)$, but do not respect the group structure of $R^*$, since the target $F_{n+1}\big(R^*,(1+I)^*\big)$ has no relations except commutativity. However, the composite maps $q_{n+1}\circ A_j$ respect both group structures, due to the multiplicative relations in $K_{n+1} ^{\tn{\fsz{M}}}(R,I)$.  Second, each map $A_j$ is injective.  Indeed, an element of $F_{n+1}\big(R^*,(1+I)^*\big)$ can belong to the image of $A_j$ only if the $j$th entries of its factors coincide, in which case its inverse image in $F_{n}\big(R^*,(1+I)^*\big)\times R^*$, if it exists, is uniquely defined by extracting the common $j$th entry.   The inverse maps $A_j^{-1}$ are therefore well-defined on the images $\tn{Im}(A_j)$.  It is important to note that the maps $A_j^{-1}$ are maps on {\it products of $(n+1)$-tuples and their inverses}, rather than merely maps on $(n+1)$-tuples. For this reason, the image of a single $(n+1)$-tuple $(\bar{r})$ under $A_j^{-1}$ is written as $A_j^{-1}\big((\bar{r})\big)$, rather than $A_j^{-1}(\bar{r})$.\\

\begin{defi}\label{defiPhinj}For an $(n+1)$-tuple $(\bar{r})=(r_1,...,r_{n+1})$ belonging to $\tn{Im}(A_j)\subset F_{n+1}\big(R^*,(1+I)^*\big)$, define 
\[\Phi_{n+1,j}\big((\bar{r})\big):= \gamma\circ\beta\circ {A_j ^{-1}}\big((\bar{r})\big),\]
and extend $\Phi_{n+1,j}$ to products of $(n+1)$-tuples sharing the same $j$th entry by sending products in $F_{n+1}\big(R^*,(1+I)^*\big)$ to sums in $\Omega_{R,I}^{n}/d\Omega_{R,I}^{n-1}$. 
\end{defi}
The following ``patching lemma" enables the definition of the ``global map" $\Phi_{n+1}$ in definition \hyperref[defipatching]{\ref{defipatching}} below:\\ 


\begin{lem}\label{lempatchingphi} $\Phi_{n+1,j}=(-1)^{k-j}\Phi_{n+1,k}$ on the intersection $\tn{Im}(A_j)\cap \tn{Im}(A_k)\subset F_{n+1}\big(R^*,(1+I)^*\big)$ for any $1\le j<k\le n$.
\end{lem}
\begin{proof} Since both $\Phi_{n+1,j}$ and $\Phi_{n+1,k}$ send products in $\tn{Im}(A_j)\cap \tn{Im}(A_k)$ to sums in $\Omega_{R,I}^{n}/d\Omega_{R,I}^{n-1}$, it suffices to prove the statement of the lemma for a single generic $(n+1)$-tuple $(\bar{r})=(r_0,...,r_{n})$ in $\tn{Im}(A_j)\cap \tn{Im}(A_k)$.   Such an $(n+1)$-tuple satisfies the condition that the $n$-tuples $(r_0,...,\hat{r_j},...,r_{n})$ and $(r_0,...,\hat{r_k},...,r_{n})$, given by deleting its $j$th and $k$th entries, respectively, belong to $F_{n}\big(R^*,(1+I)^*\big)$.  Define an $(n-1)$-tuple $(\bar{r}')$ by deleting {\it both} entries $r_j$ and $r_k$ from $(\bar{r})$, as follows:\footnotemark\footnotetext{The purpose of isolating and renaming the $(n-1)$-tuple $(\bar{r}')$, even though all its entries come from $(\bar{r})$, is to avoid numbering issues later in the proof.}  
\[(\bar{r}'):=(r_0,...,\hat{r_j},...,\hat{r_k},...,r_{n})\in F_{n-1}(R^*).\]
 By anticommutativity, the Steinberg symbols $\{\bar{r}',r_j\}$ and $\{\bar{r}',r_k\}$, defined by appending the deleted entries $r_j$ and $r_k$ onto $\bar{r}'$, respectively, belong to $K_n ^{\tn{\fsz{M}}}(R,I)$.  Using the splitting $R=S\oplus I$, as in lemma \hyperref[lemrelativegenerators]{\ref{lemrelativegenerators}} above, each entry $r_l$ of $(\bar{r})$ may be factored into a product of the form $r_l=s_l(1+i_l)$, where $s_l$ belongs to $S^*$, and $i_l$ belongs to $I$. There then exist factorizations in Milnor $K$-theory:
\begin{equation}\label{equmilnorfactorizations}\begin{array}{lcl}\{\bar{r}',r_j\}&=&\{\bar{s}',1+i_j\}\displaystyle\prod_{l=0} ^{n-2}\{\bar{r_l}',r_j\} \hspace{.5cm}\tn{and}\\
\{\bar{r}',r_k\}&=&\{\bar{s}',1+i_k\}\displaystyle\prod_{l=0} ^{n-2}\{\bar{r_l}',r_k\},\end{array}\end{equation}
where $(\bar{s}'):=(s_0,...,\hat{s_j},...,\hat{s_k},...,s_{n})$, and where $(\bar{r_l}')=(s_0',...s_{l-1}',1+i_l',r_{l+1}',...,r_{n-2}')$ has its $l$th entry in $(1+I)^*$.\footnotemark\footnotetext{{\it A priori,} the first product in equation \hyperref[equmilnorfactorizations]{\ref{equmilnorfactorizations}} has an additional factor $\{\bar{s}',s_j\}$, and the second product has an additional factor $\{\bar{s}',s_k\}$, but both factors are trivial since $(\bar{r})\in F_{n+1}\big(R^*,(1+I)^*\big)$.  For example, for $n=3$, the Steinberg symbol $\{r_1,r_2,r_3\}$ factors as $\{s_1,s_2,s_3\}\{s_1,s_2,1+i_3\}\{s_1,1+i_2,r_3\}\{1+i_1,r_2,r_3\}$, where the first factor is trivial.}  By anticommutativity and the definition of $\Phi_{n+1,j}$, it follows that
\begin{equation}\label{equPhij}\begin{array}{lcl}\Phi_{n+1,j}\big((\bar{r})\big)&=&\displaystyle(-1)^{n-k}\phi_{n}\big(\{\bar{r}',r_k\}\big)\wedge\frac{dr_j}{r_j}\\&=&\displaystyle(-1)^{k+1}\log(1+i_k)\frac{d\bar{s}'}{\Pi_{\bar{s}'}}\wedge\frac{dr_j}{r_j}+(-1)^{n-k}\omega\wedge\frac{dr_k}{r_k}\wedge\frac{dr_j}{r_j},\end{array}\end{equation}
where 
\[\omega:=\sum_{l=0} ^{n-2}(-1)^{l}\log(1+i_l')\frac{ds_0'}{s_0'}\wedge...\wedge\frac{ds_{l-1}'}{s_{l-1}'}\wedge
\frac{dr_{l+1}'}{r_{l+1}'}\wedge...\wedge\frac{dr_{n-2}'}{r_{n-2}'},\]
and where $\Pi_{\bar{s}'}$ is the product of the entries in $\bar{s}'$.

Similarly,
\begin{equation}\label{equPhik}\begin{array}{lcl}\Phi_{n+1,k}\big((\bar{r})\big)&=&\displaystyle(-1)^{n-j-1}\phi_{n}(\{\bar{r}',r_j\})\wedge\frac{dr_k}{r_k},\\&=&\displaystyle(-1)^{j}\log(1+i_j)\frac{d\bar{s}'}{\Pi_{\bar{s}'}}\wedge\frac{dr_k}{r_k}+(-1)^{n-j-1}\omega\wedge\frac{dr_j}{r_j}\wedge\frac{dr_k}{r_k}.\end{array}\end{equation}

The terms involving $\omega$ in equations \hyperref[equPhij]{\ref{equPhij}} and \hyperref[equPhik]{\ref{equPhik}} differ by the desired factor of $(-1)^{k-j}$ (note that $dr_j/r_j$ and $dr_k/r_k$ appear in the opposite order in the two equations).  It remains to show that the differential
\[\sigma:=\log(1+i_k)\frac{d\bar{s}'}{\Pi_{\bar{s}'}}\wedge\frac{dr_j}{r_j}+\log(1+i_j)\frac{d\bar{s}'}{\Pi_{\bar{s}'}}\wedge\frac{dr_k}{r_k},\]
is exact.  

\comment{
\color{red}
The point here is that multiplying the second differential by $(-1)^{k-j}$ and subtracting the two, as per the statement of the lemma, gives the differential 
\[(-1)^{k+1}\log(1+i_k)\frac{d\bar{s}'}{\Pi_{\bar{s}'}}\wedge\frac{dr_j}{r_j}-(-1)^{k-j}(-1)^{j}\log(1+i_j)\frac{d\bar{s}'}{\Pi_{\bar{s}'}}\wedge\frac{dr_k}{r_k},\]
which simplifies to 
\[(-1)^{k+1}\Bigg(\log(1+i_k)\frac{d\bar{s}'}{\Pi_{\bar{s}'}}\wedge\frac{dr_j}{r_j}+\log(1+i_j)\frac{d\bar{s}'}{\Pi_{\bar{s}'}}\wedge\frac{dr_k}{r_k}\Bigg),\]
and it is the differential in parentheses that we want to be exact. 
\color{black}
}

Since $\{\bar{r}',r_j\}$ and $\{\bar{r}',r_k\}$ belong to $K_n ^{\tn{\fsz{M}}}(R,I)$, their projections $\{\bar{s}',s_j\}$ and $\{\bar{s}',s_k\}$ in $K_n ^{\tn{\fsz{M}}}(S)$ are trivial.  Applying the canonical $d\log$ map from lemma \hyperref[lemdlog]{\ref{lemdlog}} to these projections yields
\begin{equation}\label{equdlog}\frac{d\bar{s}'}{\Pi_{\bar{s}'}}\wedge\frac{ds_j}{s_j}=\frac{d\bar{s}'}{\Pi_{\bar{s}'}}\wedge\frac{ds_k}{s_k}=0.\end{equation}
Hence, the differentials
\begin{equation}\begin{array}{lcl}\tau_1&:=&\displaystyle\log(1+i_k)\frac{d\bar{s}'}{\Pi_{\bar{s}'}}\wedge\frac{ds_j}{s_j}\hspace*{.5cm}\tn{and}\\\tau_2&:=&\displaystyle\log(1+i_j)\frac{d\bar{s}'}{\Pi_{\bar{s}'}}\wedge\frac{ds_k}{s_k},\end{array}\end{equation}
given by multiplying the differentials in equation \hyperref[equdlog]{\ref{equdlog}} by the appropriate logarithms, vanish.  Therefore, the differential 
\begin{equation}\begin{array}{lcl}\sigma&=&\sigma-\tau_1-\tau_2\\
&=&\displaystyle\log(1+i_k)\frac{d\bar{s}'}{\Pi_{\bar{s}'}}\wedge d\log(1+i_j)+\log(1+i_j)\frac{d\bar{s}'}{\Pi_{\bar{s}'}}\wedge d\log(1+i_k)\\&=&\displaystyle(-1)^{n-1}d\Big(\log\big((1+i_j)(1+i_k)\big)\frac{d\bar{s}'}{\Pi_{\bar{s}'}}\Big),\end{array}\end{equation}
is exact.  
\end{proof}


A ``global map" $\Phi_{n+1}$ from the subgroup of $F_{n+1}\big(R^*,(1+I)^*\big)$ generated by the union $\bigcup_{j=1}^{n+1}\tn{Im}(A_j)$ to $\Omega_{R,I}^{n}/d\Omega_{R,I}^{n-1}$ may now be defined by patching the maps $\Phi_{n+1,j}$ together, using lemma \hyperref[lempatchingphi]{\ref{lempatchingphi}}.  This procedure is analogous to the familiar procedure of patching together maps defined locally on open subsets of manifolds or schemes, to obtain a global map.\\ 

\begin{defi}\label{defipatching} For an $(n+1)$-tuple $(\bar{r})$ belonging to the the union $\bigcup_{j=1}^{n+1}\tn{Im}(A_j)$, define 
\[\Phi_{n+1}\big((\bar{r})\big):=(-1)^{n+1-j}\Phi_{n+1,j}\big((\bar{r})\big),\]
whenever the right-hand-side is defined, and extend to $\Phi_n$ to the subgroup of $F_{n+1}\big(R^*,(1+I)^*\big)$ generated by $\bigcup_{j=1}^{n+1}\tn{Im}(A_j)$ by taking inverses to negatives and multiplication to addition in $\Omega_{R,I}^{n}/d\Omega_{R,I}^{n-1}$. 
\end{defi}
The map $\Phi_{n+1}$ is a well-defined group homomorphism from $\bigcup_{j=1}^{n+1} \tn{Im}(A_j)$ to $\Omega_{R,I} ^n/d\Omega_{R,I} ^{n-1}$, by lemma \hyperref[]{\ref{lempatchingphi}}.  The choice of notation for $\Phi_{n+1}$ is a deliberate reflection of the fact that this map plays the same role as the maps $\Phi_{m+1}$ for $1\le m\le n-1$, introduced in definition \ref{defprelimmapsphi} above. However, whereas the maps $\Phi_{m+1}$ are defined in terms of the maps $\phi_{m+1}$, whose existence was assumed by induction, the situation here is the reverse; $\Phi_{n+1}$ is used to define $\phi_{n+1}$ below.  The image of $\bigcup_{j=1}^{n+1} \tn{Im}(A_j)$ under the quotient homomorphism $q_{n+1}$ generates the relative Milnor $K$-group $K_{n+1} ^{\tn{\fsz{M}}}(R,I)$, since any $(n+1)$-tuple $(\bar{r})=(r_0,...,r_{n})$ in $F_{n+1}(R^*)$ with at least one entry in $(1+I)^*$ belongs to $\bigcup_{j=1}^{n+1} \tn{Im}(A_j)$, and since the images of these elements under the quotient map generate $K_{n+1} ^{\tn{\fsz{M}}}(R,I)$. 

It is now possible to define the desired map $\phi_{n+1}: K_{n+1} ^{\tn{\fsz{M}}}(R,I)\rightarrow \Omega_{R,I}^n/d\Omega_{R,I} ^n.$\\

\begin{defi}\label{defiphin} Let $\prod_{l\in L}\{\bar{r_l}\}$ be an element of $K_{n+1} ^{\tn{\fsz{M}}}(R,I)$, where each factor $\{\bar{r_l}\}$ belongs to the union $\bigcup_{j=1}^{n+1} \tn{Im}(A_j)$, and where $L$ is a finite index set.  For each $l\in L$, let $(\bar{r}_l)$ be the element of $F_{n+1}\big(R^*,(1+I)^*\big)$ corresponding to the Steinberg symbol $\{\bar{r_l}\}$.  Define
\[\phi_{n+1}\Big(\prod_{l\in L}\{\bar{r_l}\}\Big):=\Phi_{n+1}\Big(\prod_{l\in L}(\bar{r_l})\Big).\]
\end{defi}


The final step regarding $\phi_{n+1}$ is to show that it is a well-defined, surjective group homomorphism.\\

\begin{lem}\label{lemelementaryphi} The map $\phi_{n+1}$ is a well-defined, surjective group homomorphism $K_{n+1}^{\tn{\fsz{M}}}(R,I)\rightarrow\Omega_{R,I}^{n}/d\Omega_{R,I}^{n-1}$.  
\end{lem}

\begin{proof} To show that $\phi_{n+1}$ is well-defined, it suffices to show that $\phi_{n+1}$ maps each multiplicative relation $\{\bar{r},rr',\bar{r}'\}\{\bar{r},r,\bar{r}'\}^{-1}\{\bar{r},r',\bar{r}'\}^{-1}$ and each Steinberg relation $\{\bar{r},r,1-r,\bar{r}'\}$ to zero in $\Omega_{R,I}^{n}/d\Omega_{R,I}^{n-1}$, where each of the Steinberg symbols appearing in these relations is assumed to have at least one entry in $(1+I)^*$.   By definition \hyperref[defiphin]{\ref{defiphin}}, this is equivalent to showing that $\Phi_{n+1}$ maps the corresponding elements $(\bar{r},rr',\bar{r}')(\bar{r},r,\bar{r}')^{-1}(\bar{r},r',\bar{r}')^{-1}$ and $(\bar{r},r,1-r,\bar{r}')$ in $F_{n+1}\big(R^*,(1+I)^*\big)$ to zero.  By lemma \hyperref[lempatchingphi]{\ref{lempatchingphi}}, it suffices to show that these elements map to zero under any map $\Phi_{n+1,j}$ whose domain contains them.   By definition \hyperref[defiPhinj]{\ref{defiPhinj}}, these elements belong to the domain of $\Phi_{n+1,j}$ if and only if the elements of $F_n(R^*)$ given by deleting their $j$th entries belong to the subgroup $F_n\big(R^*,(1+I)^*\big)$.  By definition \hyperref[defprelimmapsphi]{\ref{defprelimmapsphi}}, this is true if and only if the corresponding images in Milnor $K$-theory under the quotient map $q_n$ belong to $K_{n}^{\tn{\fsz{M}}}(R,I)$.   But choosing $j$ to be any of the barred entries produces the identity element $1$ in $K_{n}^{\tn{\fsz{M}}}(R,I)$, since the resulting products of symbols are automatically relations.  Furthermore, the corresponding map $\Phi_{n+1,j}$ sends the required elements to zero by definition, since $\phi_{n}(1)=0$.  

The map $\phi_{n+1}$ is a group homomorphism by construction, since $\Phi_{n+1}$ is defined to respect the group structure in definition \hyperref[defipatching]{\ref{defipatching}}.  To prove that $\phi_{n+1}$ is surjective, it suffices to show that any element of the form 
$rd\bar{r}d\bar{r}'$ in $\Omega_{R,I}^{n}/d\Omega_{R,I}^{n-1}$ belongs to $\tn{Im}(\phi_n)$, where $r,\bar{r}\in I$ and $\bar{r}'\in R^*$.  But for such an element, 
\[rd\bar{r}\wedge d\bar{r}'=\phi_{n+1}\big(\{e^{r\Pi'},e^{\bar{r}},\bar{r}'\}\big),\]
where as usual $\Pi'$ is the product of the entries of $\bar{r}'$. 
\end{proof}

\begin{example} \tn{The following example illustrates the reasoning involved in the first part of the proof of lemma \hyperref[lemelementaryphi]{\ref{lemelementaryphi}}. Let $n=2$, and consider the multiplicative relation 
\[\{r_0,r_1r_1',r_2\}\{r_0,r_1,r_2\}^{-1}\{r_0,r_1',r_2\}^{-1}\hspace*{.5cm}\tn{in}\hspace*{.5cm}K_{3}^{\tn{\fsz{M}}}(R,I).\]
Choose $j=0$; $j=2$ would work just as well.  One then needs to show that the element 
\[(r_0,r_1r_1',r_2)(r_0,r_1,r_2)^{-1}(r_0,r_1',r_2)^{-1},\] 
belongs to the domain of the map $\Phi_{3,0}$, and that 
\[\Phi_{3,0}\big((r_0,r_1r_1',r_2)(r_0,r_1,r_2)^{-1}(r_0,r_1',r_2)^{-1}\big)=0\hspace*{.5cm}\tn{in}\hspace*{.5cm}\Omega_{R,I}^{2}/d\Omega_{R,I}^{1}.\]
This relation, incidentally, is easy to compute directly by treating the definition of $\phi_3$ as a {\it fait accompli} and using the Leibniz rule, but this is irrelevant at the moment.  The condition on the domain follows from the obvious fact that 
\[(r_0,r_1r_1',r_2)(r_0,r_1,r_2)^{-1}(r_0,r_1',r_2)^{-1}=A_1\big((r_1r_1',r_2)(r_1,r_2)^{-1}(r_1',r_2)^{-1}, r_0\big).\]
By the definition of $\Phi_{3,0}$, it follows that:
\[\Phi_{3,0}\big((r_0,r_1r_1',r_2)(r_0,r_1,r_2)^{-1}(r_0,r_1',r_2)^{-1}\big)=\gamma\circ\beta\circ A_1^{-1}\big((r_0,r_1r_1',r_2)(r_0,r_1,r_2)^{-1}(r_0,r_1',r_2)^{-1}\big)\]
\[=\phi_2\big(\{r_1r_1',r_2\}\{r_1,r_2\}^{-1}\{r_1',r_2\}^{-1}\big)\wedge\frac{dr_0}{r_0}\]
\[=\phi_2\big(1\big)\wedge\frac{dr_0}{r_0}=0.\]}
\end{example}

\subsection{Definition and Analysis of the Map $\displaystyle\psi_{n+1}:\Omega_{R,I} ^n/d\Omega_{R,I} ^{n-1}\rightarrow K_{n+1} ^{\tn{\fsz{M}}}(R,I)$}\label{subsectionanalysispsi}

As in the case of $\phi_{n+1}$, I will define $\psi_{n+1}$ in several steps.  Recall that the induction hypothesis assumes the existence of isomorphisms
\[\psi_{m+1}:\frac{\Omega_{R,I} ^m}{d\Omega_{R,I} ^{m-1}}\rightarrow K_{m+1} ^{\tn{\fsz{M}}}(R,I)\]
\[rd\bar{r}d\bar{r}'\mapsto\{e^{r\Pi'},e^{\bar{r}},\bar{r}'\},\]
for $r,\bar{r}\in I$ and $\bar{r}'\in R^*$ for all $1\le m<n$, appearing in equation \hyperref[equinductionpsi]{\ref{equinductionpsi}} above.\\ 

\begin{defi}\label{defprelimmapspsi} Let $m$ be a positive integer.  
\begin{enumerate}
\item Let $F_{m+1}(R)$ be the free abelian group generated by ordered $(m+1)$-tuples $(r_0,...,r_m)$ of elements of $R$.  
\item Let $Q_{m+1}$ be the quotient homomorphism
\[Q_{m+1}:F_{m+1}(R)\rightarrow\frac{\Omega_R ^{m}}{d\Omega_R ^{m-1}}\]
\begin{equation}\label{equQuotient}(r_0,...,r_{m})\mapsto r_0dr_1\wedge...\wedge dr_{m}.\end{equation}
\item Let $F_{m+1}(R,I)$ be the preimage in $F_{m+1}(R)$ of the group $\Omega_{R,I} ^{m}/d\Omega_{R,I} ^{m-1}\subset\Omega_R ^{m}/d\Omega_R ^{m-1}$.
\item For $1\le m< n-1$, let $\Psi_m$ be the composition
\begin{equation}\label{equPsimplusone}\Psi_{m+1}:=\psi_{m+1}\circ Q_{m+1}:F_{m+1}(R,I)\rightarrow K_{m+1}^{\tn{\fsz{M}}}(R,I).\end{equation}
\end{enumerate}
\end{defi}

An $(m+1)$-tuple $(r_0,...,r_m)$ satisfying the condition that at least one of its entries belongs to $I$ is automatically an element of $F_{m+1}(R,I)$.  If $1\le m<n$, then specifying such an element $r_j$ allows the image of $(r_0,...,r_m)$ under $\Psi_{m+1}$ to be expressed explicitly in terms of Steinberg symbols.  However, this is more complicated than the analogous case of $\Phi_{m+1}$, discussed in section \hyperref[subsectionanalysisphi]{\ref{subsectionanalysisphi}} above.  This is because the remaining elements $r_k$ for $k\ne j$ are generally not units, while equation \hyperref[equinductionpsi]{\ref{equinductionpsi}} specifies the images $\psi_{m+1}(rd\bar{r}\wedge d\bar{r}')$ only when for $r,\bar{r}\in I$ and $\bar{r}'\in R^*$.   Hence, it is generally necessary to use lemma \hyperref[lemstability]{\ref{lemstability}} to write entries of $(r_0,...,r_m)$ which are neither units nor elements of $I$ as sums of units, then express the differential $Q_{m+1}\big((r_0,...,r_m)\big)=r_0dr_1\wedge...\wedge dr_m$ as a sum whose individual terms involve only units and elements of $I$.  Then $\psi_{m+1}$ may be applied to obtain $\Psi_{m+1}\big((r_0,...,r_m)\big)$.\\ 

\begin{example} \tn{Suppose that $m=2$, and consider the triple $(r_0,r_1,r_2)$, where $r_0$ belongs to $I$, $r_1$ belongs to $R^*$, and $r_2$ belongs to neither.  Writing $r_2$ as a sum of two units $r_2'+r_2''$ permits the computation
\[\begin{array}{lcl}\vspace*{.2cm}\Psi_{3}\big((r_0,r_1,r_2)\big)&=&\psi_3(r_0dr_1\wedge dr_2)\\
&=&\vspace*{.2cm}\psi_3\big(r_0dr_1\wedge d(r_2'+r_2'')\big)\\
&=&\vspace*{.2cm}\psi_3(r_0dr_1\wedge dr_2'+r_0dr_1\wedge dr_2'')\\
&=&\vspace*{.2cm}\psi_3(r_0dr_1\wedge dr_2')\psi_3(+r_0dr_1\wedge dr_2'')\\
&=&\{e^{r_0r_1r_2'},r_1,r_2'\}\{e^{r_0r_1r_2''},r_1,r_2''\}.\end{array}\]}
\end{example}

It will be useful to consider diagrams of the form 
\[\begin{CD}
\displaystyle\frac{\Omega_{R,I} ^{n}}{d\Omega_{R,I} ^{n-1}}
@<Q_{n+1}<<F_{n+1}(R,I)
@<\Gamma_j<<F_n(R,I)\times R
@<\sigma<<F_n(R,I)\times (R^*)^2
@>\ee>>K_{n+1} ^{\tn{\fsz{M}}}(R,I),
\end{CD}\]
where the maps $\Gamma_j$, $\sigma$, and $\ee$ are defined as follows:\\

\begin{defi}\label{defiGammajsigmaepsilon} Let $j$ be an integer between $0$ and $n-1$ inclusive. 
\begin{enumerate}
\item Let $\Gamma_j$ be the map converting each $n$-tuple $(\bar{r})$ in $F_n(R,I)$ to an $(n+1)$-tuple in $F_{n+1}(R,I)$, by inserting an element of $R$ between the $j-1$ and $j$th entries of $(\bar{r})$.  Extend $\Gamma_j$ to inverses and products of $n$-tuples to produce products of $(n+1)$-tuples and their inverses sharing the same $j$th entries.  $\Gamma_j$ plays a role directly analogous to the map $A_j$ defined in definition \hyperref[defiAjbetagamma]{\ref{defiAjbetagamma}} above. 
\item Let $\sigma$ be the map
\begin{equation}\label{equsigma}\big((\bar{r}),(u,v)\big)\mapsto\big((\bar{r}),u+v\big).\end{equation}
\item Let $\ee$ be the map
\begin{equation}\label{equvarepsilon}\big((r,\bar{r}),(u,v)\big)\mapsto\big(\Psi_{n}(ur,\bar{r})\times \{u\}\big)\big(\Psi_{n}(vr,\bar{r})\times \{v\}\big).\end{equation}
\end{enumerate}
\end{defi}
Recall that $\times$ denotes multiplication in the Milnor $K$-ring $K_{*}^{\tn{\fsz{M}}}(R)$.  Since the multiplicative group $R^*$ is the first Milnor $K$-group $K_{1} ^{\tn{\fsz{M}}}(R)$, which is the first graded piece of the Milnor $K$-ring $K_{*} ^{\tn{\fsz{M}}}(R)$, the elements $u$ and $v$ may be viewed either as elements of $R^*$ or as elements of $K_{*} ^{\tn{\fsz{M}}}(R)$.  Writing $u$ and $v$ as Steinberg symbols $\{u\}$ and $\{v\}$ on the right-hand side of equation \hyperref[equvarepsilon]{\ref{equvarepsilon}}, emphasizes the latter view, since these elements are to be multiplied on the left in $K_{*} ^{\tn{\fsz{M}}}(R)$ by $\Psi_{n-1}(ur,\bar{r})$ and $\Psi_{n-1}(vr,\bar{r})$.  

It is straightforward to verify that these maps are well-defined, and that the maps $\Gamma_j$ are injective.

Elements of $S$, and hence of $R$, may be decomposed into sums of units under appropriate stability and invertibility assumptions, as shown in lemma \hyperref[lemstability]{\ref{lemstability}} above.  The following lemma facilitates the use of this result in lemma \hyperref[lempatchingpsi]{\ref{lempatchingpsi}} below. \\ 


\begin{lem}\label{lemuvUV}Let $u$, $v$, $U$, and $V$ belong to $R^*$, and suppose that $u+v=U+V$.  Then for any $(\bar{r})\in F_n(R,I)$, 
\begin{equation}\label{equuvUV}\ee\big((\bar{r}),(u,v)\big)=\ee\big((\bar{r}),(U,V)\big).\end{equation}
\end{lem}
\begin{proof}Writing $(\bar{r})=(r,\bar{r}')$ to distinguish the first element, 
\[\begin{array}{lcl} \Psi_{n}(ur,\bar{r}')&=&\displaystyle\prod_l\psi_{n}(ur_ld\bar{r}_l\wedge d\bar{r}_l')\\
&=&\displaystyle\prod_l\{e^{ur_l\Pi_l'},e^{\bar{r}_l},\bar{r}_l'\},\end{array}\]
where $\sum_lr_ld\bar{r}_l\wedge d\bar{r}_l'$ is a decomposition of $rd\bar{r}'$ such that 
$r_l,\bar{r}_l\in I$ and $\bar{r}_l'\in R^*$.  Such a decomposition exists by lemma \hyperref[lemstability]{\ref{lemstability}} because $R$ is $2$-fold stable and $2$ is invertible in $R$. Similar formulas apply for $v$, $U$, and $V$. Thus
\[\begin{array}{lcl}\vspace*{.2cm}\ee\big((\bar{r}),(u,v)\big)&=&(\Psi_{n}\big(ur,\bar{r}')\times \{u\}\big)\big(\Psi_{n}(vr,\bar{r}')\times \{v\}\big)\\
&=&\displaystyle\prod_{l}\{e^{ur_l\Pi_l'},e^{\bar{r}_l},\bar{r}_l',u\}\prod_{l}\{e^{vr_l\Pi_l'},e^{\bar{r}_l},\bar{r}_l',v\}\\
&=&\displaystyle\Bigg(\prod_l\{e^{\bar{r}_l},\bar{r}_l'\}\times\big(\{e^{ur_l\Pi_l'},u\}\{e^{vr_l\Pi_l'},v\}\big)\Bigg)^{(-1)^{n-1}},\end{array}\]
where the exponent comes from moving the elements $e^{ur_l\Pi_l'}$ and $e^{vr_l\Pi_l'}$  to the right across the $n-1$ elements $e^{\bar{r}_l}$ and $\bar{r}_l$ before factoring out the symbols $\{e^{ur_l\Pi_l'},u\}$ and $\{e^{vr_l\Pi_l'},v\}$, which belong to $K_{2+1} ^{\tn{\fsz{M}}}(R)$.  Similarly,
\[\ee\big((\bar{r}),(U,V)\big)=\Bigg(\prod_l\{e^{\bar{r}_l},\bar{r}_l'\}\times(\{e^{Ur_l\Pi_l'},U\}\{e^{Vr_l\Pi_l'},V\}\big)\Bigg)^{(-1)^{n-1}}.\]
But for each $l$,
\[\begin{array}{lcl}\vspace*{.2cm}\{e^{ur_l\Pi_l'},u\}\{e^{vr_l\Pi_l'},v\}&=&\displaystyle\psi_2\circ\phi_2\big(\{e^{ur_l\Pi_l'},u\}\{e^{vr_l\Pi_l'},v\}\big)\\
&=&\vspace*{.2cm}\displaystyle\psi_2\Big(ur_l\Pi_l'\frac{du}{u}+vr_i\Pi_l'\frac{dv}{v}\Big)\\
&=&\vspace*{.2cm}\displaystyle\psi_2\big(r_l\Pi_l'd(u+v)\big)=\psi_2\big(r_l\Pi_l'd(U+V)\big)\\
&=&\displaystyle\{e^{Ur_l\Pi_l'},U\}\{e^{Vr_l\Pi_l'},V\}.\end{array}\]
Therefore, $\ee\big((\bar{r}),(u,v)\big)=\ee\big((\bar{r}),(U,V)\big)$, as claimed.
\end{proof}


The next step is to define maps  $\Psi_{n+1,j}$ analogous to the maps  $\Phi_{n+1,j}$ appearing in section \hyperref[subsectionanalysispsi]{\ref{subsectionanalysispsi}} above.\\

\begin{defi}\label{defiPsinplusonej}For a generator $(r_0,...,r_n)$ of $F_{n+1}(R,I)$, satisfying the condition that
$(r_0,...,\hat{r_j},...,r_n)\in F_{n}(R,I)$, define $\Psi_{n+1,j}\big((r_0,...,r_n)\big)$ to be the composition
$\ee\circ\sigma^{-1}\circ\Gamma_j ^{-1}\big((r_0,...,r_n)\big)$, and extend to products of $(n+1)$-tuples sharing 
the same $j$th entry by preserving multiplication. 
\end{defi}
To see that $\Psi_{n+1,j}$ is well defined, note that $\Gamma_j$ is injective, and although $\sigma$ is not injective, 
different preimages under $\sigma$ map to the same element of $K_{n+1} ^{\tn{\fsz{M}}}(R,I)$ under $\ee$ by 
lemma \hyperref[lemuvUV]{\ref{lemuvUV}}.\\

\begin{example}   \tn{Let $n=2$ and $j=1$, and consider the $3$-tuple $(r_0,r_1,r_2)$, where for simplicity I will assume that $r_0\in I$ and $r_2\in R^*$.   Then 
\[\begin{array}{lcl}\vspace*{.2cm}\Psi_{3,1}\big((r_0,r_1,r_2)\big)&=&\ee\circ\sigma^{-1}\circ\Gamma_3^{-1}\big((r_0,r_1,r_2)\big)\\
&=&\vspace*{.2cm}\ee\circ\sigma^{-1}\big((r_0,r_2),r_1\big)\\
&=&\vspace*{.2cm}\ee\big((r_0,r_2),(u_1,v_1)\big),\\
&=&\vspace*{.2cm}\big(\Psi_{2}\big(u_1r_0,r_2)\times \{u_1\}\big)\big(\Psi_{2}(v_1r_0,r_2)\times \{v_1\}\big)\\
&=&\{e^{u_1r_0r_2},r_2,u_1\}\{e^{v_1r_0r_2},r_2,v_1\},\end{array}\]
where $u_1+v_1$ is any decomposition of $r_1$ into a sum of units. }
\end{example}


The following lemma is, from a computational perspective, the most onerous part of the proof.\\ 

\begin{lem}\label{lempatchingpsi} $\Psi_{n+1,j}=\Psi_{n+1,k}^{(-1)^{k-j}}$ on the intersection $\tn{Im}(\Gamma_j)\cap \tn{Im}(\Gamma_k)\subset F_{n+1}(R,I)$ for any $1\le j<k\le n$. 
\end{lem}
\begin{proof} I will work out the case $0<j<k\le n$; the other cases are similar.  Since both $\Psi_{n+1,j}$ and $\Psi_{n+1,k}$ send products in $\tn{Im}(\Gamma_j)\cap \tn{Im}(\Gamma_k)$ to products in $K_{n+1} ^{\tn{\fsz{M}}}(R,I)$, it suffices to prove the statement of the lemma for a single generic $(n+1)$-tuple $(\bar{r})=(r_0,...,r_n)$ of $F_{n+1}(R,I)$.  Such an $(n+1)$-tuple satisfies the condition that $(r_0,...,\hat{r_j},...r_n)$ and $(r_0,...,\hat{r_k},...r_n)$ both belong to $F_{n}(R,I)$.  Since $R$ is $2$-fold stable and $2$ is invertible in $R$, the omitted elements $r_j$ and $r_k$ may be decomposed into sums 
\begin{equation}\label{equrjrk}r_j=u_j+v_j\hspace*{.5cm}\tn{and}\hspace*{.5cm} r_k=u_k+v_k,\end{equation}
where $u_j$, $v_j$, $u_k$, and $v_k$ belong to $R^*$.  Using the splitting $R=S\oplus I$, these summands may be further decomposed as
\begin{equation}\label{equujvjukvk}u_j=a_j+m_j,\hspace*{.5cm}v_j=b_j+n_j,\hspace*{.5cm}u_k=a_k+m_k,\hspace*{.5cm}\tn{and}\hspace*{.5cm}v_k=b_k+n_k,\end{equation}
where $a_j$, $b_j$, $a_k$, and $b_k$ belong to $S^*$, and where $m_j$, $n_j$, $m_k$, and $n_k$ belong to $I$. Let $i_j=m_j+n_j$, and $i_k=m_k+n_k$ be the $I$-parts of $r_j$ and $r_k$.  Also define $\bar{r}':=(r_0,...,\hat{r_j},...,\hat{r_k},...,r_{n})$ to be the $(n-1)$-tuple given by omitting both entries $r_j$ and $r_k$ from $(\bar{r})$. 
Since $j<k$, 
\begin{equation}\label{equkminus2}\begin{array}{lcl} \Psi_{n}\big((r_0,...,\hat{r_j},...,r_{n})\big)&=&\psi_{n}\circ Q_{n}\big((r_0,...,\hat{r_j},...,r_{n})\big)\\
&=&\psi_{n}(r_0dr_1\wedge...\wedge\hat{dr_j}\wedge...\wedge dr_n)\\
&=&\psi_{n}(r_0dr_k\wedge d\bar{r}')^{(-1)^{k-2}}.\end{array}\end{equation}

\comment{
\color{red}
This is because $dr_k$ must be moved to the left across $k-2$ differentials $dr_1,...,\hat{dr_j},...,dr_{k-1}$. 
\color{black}
}

Since $(r_0,...,\hat{r_j},...r_n)\in F_{n}(R,I)$, the differential $r_0dr_k\wedge d\bar{r}'$ may be decomposed as a sum 
\begin{equation}\label{r0drkdecomp}r_0dr_k\wedge d\bar{r}'=\sum_lr_ldi_k\wedge d\bar{r}_l+\sum_\alpha r_\alpha du_k\wedge d\bar{r}_\alpha \wedge d\bar{r}_\alpha '+\sum_\alpha r_\alpha dv_k\wedge d\bar{r}_\alpha \wedge d\bar{r}_\alpha ',\end{equation}
where $r_l$ and $\bar{r}_l$ belong to $S^*$, $r_\alpha $ and $\bar{r}_\alpha $ belong to $I$, and $\bar{r}_\alpha '$ belongs to $R^*$.  

\comment{
\color{red}
To see this, first use the splitting $R=S\oplus I$ on $r_0$ and $\bar{r}'$:
\[r_0dr_k\wedge d\bar{r}'=(s_0+i_0)dr_k\wedge d(s_1+i_1)\wedge...\wedge \hat{dr_j}\wedge...\wedge\hat{dr_k}\wedge...\wedge d(s_n+i_n).\]
All terms except one involve at least one element of $I$, and the elements not in $I$ may be broken down into sums of units.  The resulting terms give the last two sums in equation \hyperref[r0drkdecomp]{\ref{r0drkdecomp}}.  The remaining term is $s_0dr_k\wedge d\bar{s}'$, where $s_0$ is the $S$-part of $r_0$, and similarly for $\bar{s}'$.  Applying the splitting to $r_k$ in this term, the summand $s_0ds_k\wedge d\bar{s}'$ vanishes, and the remaining summand involving $i_k$ may be written as the sum over $l$ after expressing elements of $S$ as sums of units. 
\color{black}
}

Thus, after manipulating some differentials, one obtains the expression
\begin{equation}\label{equbiggerkminus2}\begin{array}{lcl} \Psi_{n}\big((u_jr_0,r_1,...,\hat{r_j},...,r_{n})\big)=\psi_{n}(u_jr_0dr_k\wedge d\bar{r}')^{(-1)^{k-2}}\\
=\Big(\displaystyle\prod_l\{u_jr_l,e^{u_jr_li_k\Pi_l},\bar{r}_l\}
\prod_\alpha \{e^{u_jr_\alpha u_k\Pi_\alpha '},u_k,e^{\bar{r}_\alpha },\bar{r}_\alpha '\}
\prod_\alpha \{e^{u_jr_\alpha v_k\Pi_\alpha '},v_k,e^{\bar{r}_\alpha },\bar{r}_\alpha '\}\Big)^{(-1)^{k-2}}.\end{array}\end{equation}

Arguing in a similar manner for $v_j$ leads to the following expression for $\Psi_{n+1,j}\big((r_0,...,r_n)\big)$:
\begin{equation}\label{equbiggestkminus2}\begin{array}{lcl} \vspace*{.2cm}\Psi_{n+1,j}\big((r_0,...,r_n)\big)=\ee\circ\sigma^{-1}\circ\Gamma_j ^{-1}\big((r_0,...,r_n)\big)\\
\hspace*{0cm}\vspace*{.5cm}=\Big(\Psi_{n}\big((u_jr_0,r_1,...,\hat{r_j},...,r_{n})\big)\times \{u_j\}\Big)\Big(\Psi_{n}\big((v_jr_0,r_1,...,\hat{r_j},...,r_{n})\big)\times \{v_j\}\Big)\\
=\Big(\displaystyle\prod_l\{u_jr_l,e^{u_jr_li_k\Pi_l},\bar{r}_l,u_j\}\prod_\alpha \{e^{u_jr_\alpha u_k\Pi_\alpha '},u_k,e^{\bar{r}_\alpha },\bar{r}_\alpha ',u_j\}
\prod_\alpha \{e^{u_jr_\alpha v_k\Pi_\alpha '},v_k,e^{\bar{r}_\alpha },\bar{r}_\alpha ',u_j\}\\
\hspace*{.3cm}\displaystyle\prod_l\{v_jr_l,e^{v_jr_li_k\Pi_l},\bar{r}_l,v_j\}\prod_\alpha \{e^{v_jr_\alpha u_k\Pi_\alpha '},u_k,e^{\bar{r}_\alpha },\bar{r}_\alpha ',v_j\}
\prod_\alpha \{e^{v_jr_\alpha v_k\Pi_\alpha '},v_k,e^{\bar{r}_\alpha },\bar{r}_\alpha ',v_j\}\Big)^{(-1)^{k-2}}\end{array}\end{equation}
\[=(abcABC)^{(-1)^{k-2}},\]
where the letters $a$, $b$, $c$, $A$, $B$, and $C$, stand for the six products over $l$ or $\alpha$, in the order shown.  By similar reasoning, and recalling that $j<k$, 
\begin{equation}\label{equbiggestjminus1}\begin{array}{lcl} \vspace*{.5cm}\Psi_{n+1,k}\big((r_0,...,r_n)\big)=\ee\circ\sigma^{-1}\circ\Gamma_k ^{-1}\big((r_0,...,r_n)\big)\\
=\Big(\displaystyle\prod_l\{u_kr_l,e^{u_kr_li_j\Pi_l},\bar{r}_l',u_k\}\prod_\alpha \{e^{u_kr_\alpha u_j\Pi_\alpha '},u_j,e^{\bar{r}_\alpha },\bar{r}_\alpha ',u_k\}\prod_\alpha \{e^{u_kr_\alpha v_j\Pi_\alpha '},v_j,e^{\bar{r}_\alpha },\bar{r}_\alpha ',u_k\}\\
\hspace*{.3cm}\displaystyle\prod_l\{v_kr_l,e^{v_kr_li_j\Pi_l},\bar{r}_l',v_k\}\prod_\alpha \{e^{v_kr_\alpha u_j\Pi_\alpha '},u_j,e^{\bar{r}_\alpha },\bar{r}_\alpha ',v_k\}\prod_\alpha \{e^{v_kr_\alpha v_j\Pi_\alpha '},v_j,e^{\bar{r}_\alpha },\bar{r}_\alpha ',v_k\}\Big)^{(-1)^{j-1}}\end{array}\end{equation}
\[=(a'b'c'A'B'C')^{(-1)^{j-1}},\]
where the letters $a'$, $b'$, $c'$, $A'$, $B'$, and $C'$, stand for the six products over $l$ or $\alpha$, in the order shown.

It will suffice to show that $a'b'c'A'B'C'=(abcABC)^{-1}$, since this implies that 
\[\Psi_{n+1,j}\big((r_0,...,r_n)\big)=(abcABC)^{(-1)^{k-2}}=\Big((a'b'c'B'C')^{(-1)^{j-1}}\Big)^{(-1)^{k-j}}=\Big(\Psi_{n+1,k}\big((r_0,...,r_n)\big)\Big)^{(-1)^{k-j}}.\]

\comment{
\color{red}
The full computation is 
\[(abcABC)^{(-1)^{k-2}}=(a'b'c'A'B'C')^{(-1)^{k-1}}=(a'b'c'A'B'C')^{(-1)^{(j-1)+(k-j)}}=\Big((a'b'c'B'C')^{(-1)^{j-1}}\Big)^{(-1)^{k-j}}.\]
\color{black}
}

Now by anticommutativity,
\begin{equation}\label{bprimeequalsbinverse}b'=b^{-1},c'=B^{-1},B'=c^{-1},C'=C^{-1}.\end{equation}

\comment{
\color{red}
Indeed,
\[b=\prod_\alpha \{e^{u_jr_\alpha u_k\Pi_\alpha '},u_k,e^{\bar{r}_\alpha },\bar{r}_\alpha ',u_j\}\hspace*{.5cm}\tn{and}\hspace*{.5cm} b'=\prod_\alpha \{e^{u_kr_\alpha u_j\Pi_\alpha '},u_j,e^{\bar{r}_\alpha },\bar{r}_\alpha ',u_k\};\]
The first entries of $b$ and $b'$ are the same.  Working with $b'$, moving $u_k$ to the left across the entries $u_j,e^{\bar{r}_\alpha },\bar{r}_\alpha '$, then moving $u_j$ to the right across the entries $e^{\bar{r}_\alpha },\bar{r}_\alpha '$, leaves one extra factor of $-1$, so $b'=b^{-1}$.   Analogous arguments apply to the pairs $(c',B)$, $(B',c)$, and $(C',C)$. It follows that \[(bcBC)^{(-1)^{k-2}}=(b'c'B'C')^{(-1)^{k-1}}=(b'c'B'C')^{(-1)^{(j-1)+(k-j)}}=\Big((b'c'B'C')^{(-1)^{j-1}}\Big)^{(-1)^{k-j}}.\]
\color{black}
}

It remains to show that $aA=(a'A')^{-1}$; i.e., that $aAa'A'=1$.  Factoring out terms with repeated entries, it suffices to show that 
\[\prod_l\{r_l,e^{u_jr_li_k\Pi_l},\bar{r}_l,u_j\}\prod_l\{r_l,e^{v_jr_li_k\Pi_l},\bar{r}_l,v_j\}
\prod_l\{r_l,e^{u_kr_li_j\Pi_l},\bar{r}_l,u_k\}\prod_l\{r_l,e^{v_kr_li_j\Pi_l},\bar{r}_l,v_k\}=1.\]

\comment{
\color{red}
Indeed,
\[aAa'A'=\prod_l\{u_jr_l,e^{u_jr_li_k\Pi_l},\bar{r}_l,u_j\}\prod_l\{v_jr_l,e^{v_jr_li_k\Pi_l},\bar{r}_l,v_j\}\prod_l\{u_kr_l,e^{u_kr_li_j\Pi_l},\bar{r}_l',u_k\}\prod_l\{v_kr_l,e^{v_kr_li_j\Pi_l},\bar{r}_l',v_k\}.\]
Factoring the first entry of $a$ two factors: the first has first and last entries $u_j$ and $u_j$, respectively, and therefore vanishes.  The second has first and last entries $r_l$ and $u_j$, and is retained.  Similar arguments apply to $A$, $a'$, and $A'$.  
\color{black}
}

Applying anticommutativity and factoring in $K_\bullet^{\tn{\fsz{M}}}(R)$, it suffices to show that
\[\prod_l\big(\{e^{u_jr_li_k\Pi_l},u_j\}\{e^{v_jr_li_k\Pi_l},v_j\}
\{e^{u_kr_li_j\Pi_l},u_k\}\{e^{v_kr_li_j\Pi_l},v_k\}\big)\times\{r_l,\bar{r}_l\}=1.\]

\comment{
\color{red}
Indeed,
Here, $\{r_1,\bar{r}_l\}$ is factored out of each term after moving its entries to the far right of each Steinberg symbol.
\color{black}
}

Using the isomorphisms $\phi_2$ and $\psi_2$, applied to each left-hand factor:
\begin{equation}\label{usingphi2psi2}\begin{array}{lcl}\vspace*{.2cm}\psi_2\circ\phi_2\big(\{e^{u_jr_li_k\Pi_l},u_j\}\{e^{v_jr_li_k\Pi_l},v_j\}\{e^{u_kr_li_j\Pi_l},u_k\}\{e^{v_kr_li_j\Pi_l},v_k\}\big)\\
\vspace*{.2cm}=\psi_2(r_li_k\Pi_ldu_j+r_li_k\Pi_ldv_j+r_li_j\Pi_ldu_k+r_li_j\Pi_ldv_k)\\
\vspace*{.2cm}=\psi_2(r_li_k\Pi_ldr_j+r_li_j\Pi_ldr_k)\\
\vspace*{.2cm}=\psi_2\big(r_l\Pi_ld(i_ji_k)+r_li_k\Pi_lds_j+r_li_j\Pi_lds_k\big)\\
\vspace*{.2cm}=\psi_2(-i_ji_kd(r_l\Pi_l)+r_li_k\Pi_lda_j+r_li_k\Pi_ldb_j+r_li_j\Pi_lda_k+r_li_j\Pi_ldb_k)\\
=\{e^{i_ji_kr_l\Pi_l},r_l\Pi_l\}^{-1}\{e^{r_li_k\Pi_la_j},a_j\}\{e^{r_li_k\Pi_lb_j},b_j\}
\{e^{r_li_j\Pi_la_k},a_k\}\{e^{r_li_j\Pi_lb_k},b_k\}.\end{array}\end{equation}

\comment{
\color{red}
For second equality, combine $r_j=u_j+v_j$ and $r_k=u_k+v_k$.   For third,  writing out the expression gives
\[\psi_2\big(r_l\Pi_ld(i_ji_k)+r_li_k\Pi_lds_j+r_li_j\Pi_lds_k\big)=\psi_2\big(r_li_k\Pi_ldi_j+r_li_j\Pi_ldi_k+r_li_k\Pi_lds_j+r_li_j\Pi_lds_k\big),\]
and the first and third, and second and fourth, terms combine to give the previous expression $\psi_2(r_li_k\Pi_ldr_j+r_li_j\Pi_ldr_k)$.  For fourth, The first term uses exactness and the Leibniz rule, while the other terms break up $s_j$ and $s_k$ into sums of units. 
\color{black}
}

``Re-multiplying" in the Milnor $K$-ring $K_{*}^{\tn{\fsz{M}}}(R)$, it suffices to show that
\[\prod_l\{e^{i_ji_kr_l\Pi_l},r_l\Pi_l,r_l,\bar{r}_l\}^{-1}\{e^{r_li_k\Pi_la_j},a_j,r_l,\bar{r}_l\}
\{e^{r_li_k\Pi_lb_j},b_j,r_l,\bar{r}_l\}
\{e^{r_li_j\Pi_la_k},a_k,r_l,\bar{r}_l\}\{e^{r_li_j\Pi_lb_k},b_k,r_l,\bar{r}_l\}=1.\]
The first factor is trivial, because it can be factored in $K_{n+1} ^{\tn{\fsz{M}}}(R,I)$ into a product of terms with repeated 
entries.  For the next two factors, note that since $r_0dr_j\wedge d\bar{r}'\in\Omega_{R,I} ^{n-1}/d\Omega_{R,I} ^{n-2}$, it follows that
\[s_0ds_j\wedge d\bar{s}'=\sum_lr_lda_j\wedge d\bar{r}_l+\sum_lr_ldb_j\wedge d\bar{r}_l=0.\]
Applying $i_kd$, 
\[\sum_li_kdr_l\wedge da_j\wedge d\bar{r}_l+\sum_li_kdr_l\wedge db_j\wedge d\bar{r}_l=0.\]
Under the isomorphism $\psi_{n}$, this maps to 
\begin{equation}\label{secondthirdtrivial}\prod_l\{e^{r_li_k\Pi_la_j},a_j,r_l,\bar{r}_l\}\{e^{r_li_k\Pi_lb_j},b_j,r_l,\bar{r}_l\}=1.\end{equation}
For the final two factors, apply $i_jd$ to $s_0ds_k\wedge d\bar{s}'=0$, then apply $\psi_{n}$ to show that 
\begin{equation}\label{fourthfifthtrivial}\prod_l\{e^{r_li_j\Pi_la_k},a_k,r_l,\bar{r}_l\}\{e^{r_li_j\Pi_lb_k},b_k,r_l,\bar{r}_l\}=1.\end{equation}
This completes the proof of the case $0<j<k\le n$.   The remaining cases, in which $j=0$, are nearly identical, though slightly easier. 

\end{proof}


A ``global map" $\Psi_{n+1}$ from the subgroup of $F_{n+1}(R,I)$ generated by the union $\bigcup_{j=1}^{n+1}\tn{Im}(\Gamma_j)$ to $K_{n+1} ^{\tn{\fsz{M}}}(R,I)$ may now be defined by patching the maps $\Psi_{n+1,j}$ together, using lemma \hyperref[lempatchingpsi]{\ref{lempatchingphi}}.\\

\begin{defi}\label{defipatchingpsi} For an $(n+1)$-tuple $(\bar{r})$ belonging to the the union $\bigcup_{j=1}^{n+1}\tn{Im}(\Gamma_j)$, define 
\begin{equation}\label{equdefiPsinplusone}\Psi_{n+1}\big((\bar{r})\big):=\Psi_{n+1,j}\big((\bar{r})\big)^{(-1)^{n+1-j}},\end{equation}
whenever the right-hand-side is defined, and extend to $\Psi_n$ to the subgroup of $F_{n+1}(R,I)$ generated by $\bigcup_{j=1}^{n+1}\tn{Im}(\Gamma_j)$ by taking inverses to negatives and multiplication to multiplication in $K_{n+1} ^{\tn{\fsz{M}}}(R,I)$. 
\end{defi}

The map $\Psi_{n+1}$ is a well-defined group homomorphism from $\bigcup_{j=1}^{n+1} \tn{Im}(\Gamma_j)$ to $K_{n+1} ^{\tn{\fsz{M}}}(R,I)$, by lemma \hyperref[]{\ref{lempatchingpsi}}.  The choice of notation for $\Psi_{n+1}$ is a deliberate reflection of the fact that this map plays the same role as the maps $\Psi_{m+1}$ for $1\le m\le n-1$, introduced in definition \ref{defprelimmapspsi} above.  The image of $\bigcup_{j=1}^{n+1} \tn{Im}(\Gamma_j)$ under the quotient homomorphism $Q_{n+1}$ generates the group $\Omega_{R,I} ^n/d\Omega_{R,I} ^{n-1}$, since any $(n+1)$-tuple $(\bar{r})=(r_0,...,r_{n})$ in $F_{n+1}(R)$ with at least one entry in $I$ belongs to $\bigcup_{j=1}^{n+1} \tn{Im}(\Gamma_j)$, and since the images of these elements under the quotient map generate $\Omega_{R,I} ^n/d\Omega_{R,I} ^{n-1}$.

It is now possible to define the desired map $\psi_{n+1}: \Omega_{R,I}^n/d\Omega_{R,I} ^n\rightarrow K_{n+1} ^{\tn{\fsz{M}}}(R,I)$.\\

\begin{defi}\label{defipsinplusone} For any element $\omega\in\Omega_{R,I} ^n/d\Omega_{R,I} ^{n-1}$, write 
$\omega=\sum_ir_id\bar{r}_i$, where $r_id\bar{r_i}\in Q_{n+1}\big(\bigcup_j \tn{Im}(\Gamma_j)\big)$. Now define
\begin{equation}\label{equdefipsinplusone}\psi_{n+1}(\omega)=\Psi_{n+1}\big(\prod_i(r_i,\bar{r_i})\big).\end{equation}
\end{defi}

The next step is to show that $\psi_{n+1}$ is a well-defined group homomorphism.\\

\begin{lem}\label{lemelementarypsi} The map $\psi_{n+1}$ is a well-defined group homomorphism $\Omega_{R,I}^{n}/d\Omega_{R,I}^{n-1}\rightarrow K_{n+1}^{\tn{\fsz{M}}}(R,I)$.  
\end{lem}

\begin{proof} To show that $\psi_{n+1}$ is well-defined, it suffices to show that $\psi_{n+1}$ maps the relations in lemma \hyperref[lemKahlerrelations]{\ref{lemKahlerrelations}} to the identity in $K_{n+1}^{\tn{\fsz{M}}}(R,I)$.  To streamline the notation, let $\ms{R}$ be such a relation.   Following similar reasoning to that used in the proof of lemma \hyperref[lemelementaryphi]{\ref{lemelementaryphi}}, it suffices to show that $\Psi_{n+1,j}(\ms{R})$ is defined, and equal to $1$, for some $j$.  Now $\Psi_{n+1,j}(\ms{R})$ is defined whenever $j\ne l$ for the additivity relation, whenever $j\ne 0,l$ for the Leibniz rule, and whenever $j\ne l, l+1$ for anticommutativity. In all cases, $\Psi_{n+1,j}(\ms{R})=1$, since omitting the $j$th entry yields a relation in $F_{n}(R,I)$.  The map $\psi_{n+1}$ is a group homomorphism by construction, since $\Psi_{n+1}$ is defined to respect the group structure in definition \hyperref[defipatchingpsi]{\ref{defipatchingpsi}}.  
\end{proof}

The following lemma is the final step in the proof of the theorem:\\

\begin{lem} The maps $\phi_{n+1}$ and $\psi_{n+1}$ are inverse isomorphisms.  
\end{lem}

\begin{proof}It suffices to show this on sets of generators.  $K_{n+1} ^{\tn{\fsz{M}}}(R,I)$ is generated by 
elements of the form $\{r_0,...,r_{n}\}$ with $r_i\in S^*\cup(1+I)^*$ and at least one $r_i$ in $(1+I)^*$.  
By anticommutativity, such an element can be written as $\{r,\bar{r}\}$, where $r\in(1+I)^*$ and $\bar{r}\in R^*$.  
For such an element,
\begin{equation}\label{inversepsiphi}\psi_{n+1}\circ\phi_{n+1}\big(\{r,\bar{r}\}\big)=\psi_{n+1}\Big(\log(r)\frac{d\bar{r}}{\Pi}\Big)=\{e^{\frac{\log(r)}{\Pi}\Pi},\bar{r}\}=\{r,\bar{r}\}.\end{equation}
Finally, $\Omega_{R,I} ^n/d\Omega_{R,I} ^{n-1}$ is generated by elements of the form $rdr_1\wedge...\wedge dr_n$ 
with $r_i\in S^*\cup I$ and at least one $r_i$ in $I$.  By anticommutativity and exactness, such an element can 
be written as $rd\bar{r}\wedge d\bar{r}'$, where $r,\bar{r}\in I$ and $\bar{r}'\in R^*$.  For such an element,
\begin{equation}\label{inversepsiphi}\phi_{n+1}\big(\psi_{n+1}(rd\bar{r}\wedge d\bar{r}')\big)=\phi_{n+1}\big(\{e^{r\Pi'},e^{\bar{r}},\bar{r}'\}\big)= \log\big(e^{r\Pi'}\big)d\log(e^{\bar{r}})\wedge \frac{d\bar{r}'}{\Pi'}=rd\bar{r}\wedge d\bar{r}'.\end{equation}
\end{proof}

\section{Discussion and Applications}\label{sectiondiscussion}

\subsection{Green and Griffiths: Infinitesimal Structure of Chow Groups}\label{subsectionGG}

The original motivation for this paper arose from an attempt to understand Green and Griffiths' suggestive yet incomplete study \cite{GreenGriffithsTangentSpaces05} of the infinitesimal structure of cycle groups and Chow groups over smooth algebraic varieties.   Suppose $X$ is a smooth algebraic variety over a field $k$ containing the rational numbers.  Then Bloch's theorem \cite{BlochK2Cycles74}, extended by Quillen \cite{QuillenHigherKTheoryI72}, expresses the Chow groups of $X$ as Zariski sheaf cohomology groups of the Quillen $K$-theory sheaves\footnotemark\footnotetext{These are the sheaves associated to the presheaves $U\mapsto K_p(U)$ for open $U\subset X$.  I use $p$ here as a generic superscript; $p$ is usually $n+1$ in the context of the main theorem.}  $\ms{K}_p$ on $X$:
\begin{equation}\label{blochstheorem}
\tn{Ch}^p(X)=H_{\tn{\fsz{Zar}}}^p(X,\ms{K}_p).
\end{equation}
The general intractability of the Chow groups $\tn{Ch}_X^p$ for $p\ge 2$ makes the {\it linearization} of equation \hyperref[blochstheorem]{\ref{blochstheorem}} a problem of obvious interest, somewhat in the same spirit as the linearization of Lie groups via much simpler Lie algebras.\footnotemark\footnotetext{In view of Bloch's theorem \hyperref[blochstheorem]{\ref{blochstheorem}}, this is much more than an analogy.  In particular, the relationship between algebraic $K$-theory and cyclic homology shares many nearly-identical structural features with the relationship between Lie groups and Lie algebras.  See Loday \cite{LodayCyclicHomology98}, Chapters 10 and 11 for details.}  Following this reasoning, and skipping some details, leads to the expression
\begin{equation}\label{linearblochstheorem}
T\tn{Ch}^p(X)=H_{\tn{\fsz{Zar}}}^p(X,T\ms{K}_p,),
\end{equation}
where $T\tn{Ch}^p(X)$ is the {\it tangent group at the origin} of the Chow group $\tn{Ch}^p(X)$, and where $T\ms{K}_p$ is the {\it tangent sheaf at the origin} of the $K$-theory sheaf $\ms{K}_p$.  In this context, $T\ms{K}_p$ is the relative sheaf defined via the simplest nontrivial split nilpotent extension of the structure sheaf $\ms{O}_X$ of $X$, given by tensoring $\ms{O}_X$ with the ring of dual numbers $k[\ee]/\ee^2$ over $k$.   At a given point $x\in X$, this involves extending the local ring $S=\ms{O}_{X,x}$ to the ring $R=S\otimes_kk[\ee]/\ee^2=S[\ee]/\ee^2$, with nilpotent extension ideal $I=(\ee)$.   The terminology {\it at the origin} may be understood by noting that elements of the relative $K$-group $K_p(R,I)$ may be viewed as ``infinitesimal deformations" of the identity element in $K_p(S)$, since the canonical map $K_p(R)\rightarrow K_p(S)$ sends every element of the subgroup $K_p(R,I)\subset K_p(R)$ to the origin in $K_p(S)$.   These considerations bring the study of {\it relative} $K$-theory, and hence of Goodwillie-type theorems, squarely into the picture.   

Green and Griffiths focus on the case of $\tn{Ch}^2(X)$, where $X$ is a smooth algebraic surface over a field $k$ containing the rational numbers.   Historically, this case has provided some of the most important and surprising results in the theory of Chow groups.   The $K$-theory sheaf involved in this context is $\ms{K}_2$, and one may substitute the corresponding {\it Milnor} sheaf $\ms{K}_{2} ^{\tn{\fsz{M}}}$, since the functor $K_{2} ^{\tn{\fsz{M}}}$, defined using the na\"{i}ve tensor product definition \hyperref[defiMilnorK]{\ref{defiMilnorK}}, coincides with $K_2$ on the local rings of $X$.  The object of interest is then 
\begin{equation}\label{linearblochstheorem}
T\tn{Ch}^2(X)=H_{\tn{\fsz{Zar}}}^2(X,T\ms{K}_{2} ^{\tn{\fsz{M}}}).
\end{equation}
Equation \hyperref[linearblochstheorem]{\ref{linearblochstheorem}} represents information about an object viewed as totally intractable; namely $\tn{Ch}^2(X)$, in terms of objects viewed as elementary; namely the relative Milnor $K$-groups $K_{2} ^{\tn{\fsz{M}}}(R,I)$, which may be described in terms of K\"{a}hler differentials.  This expression provides hope for acquiring useful geometric understanding in a somewhat broader context by means of symbolic $K$-theory, avoiding as much as possible an otherwise forbidding morass of modern homotopy-theoretic constructions. 


\subsection{Similar Results involving Relative $K$-Theory and Infinitesimal Geometry}\label{subsectionsimilarresults}

{\bf Van der Kallen: an Early Computation of $K_{2} ^{\tn{\fsz{M}}}(R,I)$.}  The isomorphism $K_{2}(R,I)\cong \Omega_{R,I} ^1/dI$ of Bloch \cite{BlochK2Artinian75}, stated in theorem \hyperref[theorembloch]{\ref{theorembloch}} above under appropriate hypotheses, clearly applies in the case discussed in section \hyperref[subsectionGG]{\ref{subsectionGG}}, in which $R=S[\ee]/\ee^2$ and $I=(\ee)$, $S=R/I$ is assumed to be local, $I$ is nilpotent, and the underlying field $k$ contains $\QQ$.   Under these conditions, it is easy to show that the group $\Omega_{R,I} ^1/dI$ is isomorphic to the group $\Omega_S^1=\Omega_{S/\ZZ}^1$ of absolute K\"{a}hler differentials over $S$.  Indeed, by lemma \hyperref[lemrelativekahlergenerators]{\ref{lemrelativekahlergenerators}}, the relative group $\Omega_{R,I} ^1$ is generated by differentials of the form $\ee adb+c d\ee$ for some $a,b,c\in S$.  Hence, in the quotient $\Omega_{R,I} ^1/dI$, 
\begin{equation}\label{equvanderkallencomputation}d(c\ee)=cd\ee+\ee dc=0,\hspace*{.5cm}\tn{so}\hspace*{.5cm}cd\ee=-\ee dc,\end{equation}
 by the Leibniz rule and exactness.  This shows that $\Omega_{R,I} ^1/dI$ is generated by differentials of the form $\ee adb$, and it is easy to see that all the remaining relations come from $\Omega_{S/\ZZ}^1$.  Identifying $\ee adb$ with $adb$ then gives the isomorphism $\Omega_{R,I}^1\cong\Omega_{S/\ZZ}^1$.\footnotemark\footnotetext{In many instances, it is better to ``carry along the $\ee$," and to think of $\Omega_{R,I}^1$ as $\Omega_{S/\ZZ}^1\otimes_k(\ee)$, since the latter form generalizes in important ways.  For example, $\Omega_{S/\ZZ}^1\otimes_k(\ee)$ is replaced by $\Omega_{S/\ZZ}^1\otimes_k\mfr{m}$ for an appropriate local artinian $k$-algebra with maximal idea $\mfr{m}$ and residue field $k$ in the context of Stienstra's paper \cite{StienstraFormalCompletion83} on the formal completion of $\tn{Ch}^2(X)$}

The specific result 
\begin{equation}\label{equvanderkallenearly}K_{2}\big(S[\ee]/\ee^2,(\ee)\big)\cong \Omega_{S/\ZZ}^1,\end{equation}
 is due to Van der Kallen \cite{VanderKallenEarlyTK271}, predating the more general result of Bloch by several years.  Van der Kallen's result features prominently in the work of Green and Griffiths.\footnotemark\footnotetext{In fact, Green and Griffiths give a messy but elementary symbolic proof of Van der Kallen's result, without using Bloch's theorem; see \cite{GreenGriffithsTangentSpaces05}, appendix 6.3.1, pages 77-81.}

Sheafifying equation \hyperref[equvanderkallenearly]{\ref{equvanderkallenearly}} and substituting it into equation \hyperref[linearblochstheorem]{\ref{linearblochstheorem}} yields the expression
\begin{equation}\label{linearblochstheoremkahler}
T\tn{Ch}^2(X)=H_{\tn{\fsz{Zar}}}^2(X,\varOmega_{X/\ZZ}^1),
\end{equation}
where $\varOmega_{X/\ZZ}^1$ is the sheaf of absolute K\"{a}hler differentials on $X$.   Working primarily from the viewpoint of complex algebraic geometry, Green and Griffiths were struck by the ``mysterious" appearance of absolute differentials in this context, and much of their study \cite{GreenGriffithsTangentSpaces05} is an effort to explain the ``geometric origins" of such differentials.  The right-hand-side of equation \hyperref[linearblochstheoremkahler]{\ref{linearblochstheoremkahler}} is what Green and Griffiths call the ``formal tangent space to $\tn{Ch}^2(X)$."

Green and Griffiths generalize equation \hyperref[linearblochstheoremkahler]{\ref{linearblochstheoremkahler}} to give a definition (\cite{GreenGriffithsTangentSpaces05} equation 8.53, page 145) of the tangent space $T\tn{Ch}^p(X)$ of the $p$th Chow group of a $p$-dimensional smooth projective variety:
\begin{equation}\label{linearblochstheoremkahlerhigher}
T\tn{Ch}^p(X)=H_{\tn{\fsz{Zar}}}^p(X, \varOmega_{X/\ZZ}^{p-1}).
\end{equation}
Equation \hyperref[linearblochstheoremkahlerhigher]{\ref{linearblochstheoremkahlerhigher}} is a linearization of the corresponding case of Bloch's theorem in equation \hyperref[blochstheorem]{\ref{blochstheorem}} above.  From the viewpoint of the present paper, the sheaf $\varOmega_{X/\ZZ}^{p-1}$ in equation \hyperref[linearblochstheoremkahlerhigher]{\ref{linearblochstheoremkahlerhigher}} may be derived by sheafifying a special case of the main theorem in equation \hyperref[equmaintheorem]{\ref{equmaintheorem}}, given by setting $R=S[\ee]/\ee^2$, $I=(\ee)$, and $n=p-1$, where $S$ is taken to be the local ring at a point on $X$:
\begin{equation}\label{equspecialcasedual}K_{p} ^{\tn{\fsz{M}}}\big(S[\ee]/\ee^2, (\ee)\big)\cong\Omega_{S[\ee]/\ee^2, (\ee)} ^{p-1}/d\Omega_{S[\ee]/\ee^2, (\ee)} ^{p-2}\cong \Omega_{S/\ZZ} ^{p-1}.\end{equation}
The second isomorphism in equation \hyperref[equspecialcasedual]{\ref{equspecialcasedual}} follows easily from the Leibniz rule and exactness, as in equation \hyperref[equvanderkallencomputation]{\ref{equvanderkallencomputation}} in the case $n=1$.  In view of equation \hyperref[equspecialcasedual]{\ref{equspecialcasedual}}, the group $\Omega_{S/\ZZ} ^{p-1}$ may be identified as the tangent group at the origin of the Milnor $K$-group $K_{p} ^{\tn{\fsz{M}}}(S)$.  The problem of finding meaningful generalizations of the expression for $T\tn{Ch}^{p}(X)$ in equation \hyperref[linearblochstheoremkahlerhigher]{\ref{linearblochstheoremkahlerhigher}} in cases in which the extension ideal $I$ is more complicated than $(\ee)$ is one of the principal motivations for this paper.   I return to this point in section \hyperref[subsectiontangentfunctors]{\ref{subsectiontangentfunctors}} below. 


{\bf Stienstra: the Formal Completion of $\tn{Ch}^2(X)$.}  In the special case of algebraic surfaces, such a generalization was carried out twenty years ago by Jan Stienstra in his study \cite{StienstraFormalCompletion83} of the {\it formal completion at the origin}  $\widehat{\tn{Ch}}_X^2(A,\mfr{m})$ of the Chow group  $\tn{Ch}^2(X)$ of a smooth projective surface defined over a field $k$ containing the rational numbers.   As discussed above, extension of a $k$-algebra $S$ to the ring $S[\ee]/(\ee^2)$ of dual numbers over $S$ is the simplest nontrivial type of split nilpotent extension, and Bloch's theorem \hyperref[theorembloch]{\ref{theorembloch}} immediately gives much more.  Exploring this thread, Stienstra identifies $\widehat{\tn{Ch}}_X^2(A,\mfr{m})$ as the Zariski sheaf cohomology group 
\begin{equation}\label{equstienstra}\widehat{\tn{Ch}}_X^2(A,\mfr{m})=H_{\tn{\fsz{Zar}}}^2\Bigg(X,\frac{\varOmega_{X\otimes_k A,X\otimes_k\mfr{m}}^1}{d(\ms{O}_X\otimes_k\mfr{m})}\Bigg),\end{equation}
where $A$ is a local artinian $k$-algebra with maximal ideal $\mfr{m}$ and residue field $k$, and where $\widehat{\tn{Ch}}_X^2$ is viewed as a functor from the category of local artinian $k$-algebras with residue field $k$ to the category of abelian groups.\footnotemark\footnotetext{See Stienstra \cite{StienstraFormalCompletion83} page 366. Stienstra writes, {\it ``We may forget about $K$-theory.  Our problem has become analyzing $H^n\big(X,\varOmega_{X\otimes A,X\otimes m}/d(\ms{O}_X\otimes_k\mfr{m})\big)$, as a functor of $(A,\mfr{m})$."}  Much of Stienstra's paper consists of expressing these cohomology groups in terms of the simpler objects $H^n(X,\ms{O}_X)$, $\Omega_{A,\mfr{m}}^1$, and $H^n(X,\varOmega_{X/\ZZ}^1)$.  Similar analysis of the cohomology groups in equation \hyperref[equgeneralizedtangent]{\ref{equgeneralizedtangent}} below is a problem of obvious interest.}  This formidable-looking expression is merely a sheafification of Bloch's isomorphism $K_{2}(R,I)\cong \Omega_{R,I} ^1/dI$ in theorem \hyperref[theorembloch]{\ref{theorembloch}}, substituted into Bloch's expression for the Chow groups $\tn{Ch}^p(X)=H_{\tn{\fsz{Zar}}}^p(\ms{K}_p, X)$ in equation \hyperref[blochstheorem]{\ref{blochstheorem}}.   For example, if $A$ is the ring of dual numbers $k[\ee]/\ee^2$ over $k$, and $\mfr{m}$ is the ideal $(\ee)$, then one recovers the case studied by Green and Griffiths:
\[\widehat{\tn{Ch}}_X^2\big(k[\ee]/\ee^2, (\ee)\big)=T\tn{Ch}^2(X)=H_{\tn{\fsz{Zar}}}^2(X,\varOmega_{X/\ZZ}^1).\]


{\bf Hesselholt: Relative $K$-Theory of Truncated Polynomial Algebras.}  More recently, Lars Hesselholt has done very substantial work on the relative $K$-theory of rings with respect to nilpotent ideals and Goodwillie-type theorems.\footnotemark\footnotetext{In a formal sense, this story may be considered complete: as Hesselholt points out (\cite{HesselholtTruncated05}, page 72) {\it ``If the ideal... ... is nilpotent, the relative $K$-theory can be expressed completely in terms of the cyclic homology of Connes and the topological cyclic homology of B\"{o}kstedt-Hsiang-Madsen."}  However, one is still concerned with unwinding the latter theories in cases of particular interest.  This is the object of Hesselholt's paper \cite{HesselholtTruncated05}, and also, in a smaller way, of mine.}  Here, I cite only one of his many results that is of particular relevance to the subject of this paper.   In his paper \cite{HesselholtTruncated05}, Hesselholt computes the relative $K$-theory (not just Milnor $K$-theory) of truncated polynomial algebras; i.e., polynomial algebras of the form $S[\ee]/\ee^N$ for some integer $N\ge1$, where $S$ is commutative regular noetherian algebra over a field.  The case $N=1$ returns $S$, and the case $N=2$ gives the now-familiar extension of $S$ by the dual numbers.  Hesselholt's result may be expressed as follows:
\begin{equation}\label{equhesselholt}K_{n+1}\big(S[\ee]/\ee^N,(\ee)\big)\cong\bigoplus_{m\ge0}\big(\Omega_{S/\ZZ} ^{n-2m}\big)^{N-1}.\end{equation}
The case $N=2$ gives the expression
\[K_{n+1}\big(S[\ee]/\ee^2,(\ee)\big)\cong\Omega_{S/\ZZ} ^{n}\oplus \Omega_{S/\ZZ} ^{n-2}\oplus \Omega_{S/\ZZ} ^{n-4}\oplus...,\]
where the first summand on the right-hand side is immediately recognizable as the tangent group at the origin of Milnor $K$-theory, identified in equation \hyperref[equspecialcasedual]{\ref{equspecialcasedual}}.  The remaining summands may be viewed, roughly speaking, as representing ``tangents to the non-symbolic part of $K$-theory."  


\subsection{Generalized Tangent Functors}\label{subsectiontangentfunctors}

There exist a number of obvious ways in which one may attempt to generalize Green and Griffiths' study of the tangent group at the origin $T\tn{Ch}^2(X)$ of the Chow group $\tn{Ch}^2(X)$ of a smooth algebraic surface:
\begin{enumerate}
\item\label{higherdim} Higher-dimensional varieties may be considered, as in equation \hyperref[linearblochstheoremkahlerhigher]{\ref{linearblochstheoremkahlerhigher}}.  As Green and Griffiths point out in \cite{GreenGriffithsTangentSpaces05}, much of their work concerning $\tn{Ch}^2(X)$ applies immediately to $\tn{Ch}^n(X)$ for an $n$-dimensional variety. 
\item\label{othercodim} Codimensions different than the dimension of the variety may be studied; for example, one may examine $T\tn{Ch}^2(X)$ for a $3$-fold. 
\item\label{otherinfinitesimal} Infinitesimal information more complicated than the dual numbers may be added to the picture, as in Stienstra's paper \cite{StienstraFormalCompletion83}.  In other words, one may choose to study functors such as the formal completion $\widehat{\tn{Ch}}_X^2$, rather than merely the tangent space.\footnotemark\footnotetext{This may be viewed as an exalted analogue of studying Taylor polynomials rather than merely tangent lines.}
\item\label{otherK} More sophisticated $K$-theory may be employed, as suggested by Hesselholt's theorem, which shows that ``there is more to relative $K$-theory than relative Milnor $K$-theory," even in the simplest cases.   For deep structural reasons, the nonconnective $K$-theory of Bass and Thomason gives good formal results, but Quillen $K$-theory is inadequate. 
\item\label{finitechar} The case of positive characteristic may be considered.  
\item\label{nonsmooth} Smooth algebraic varieties may be exchanged for a more general category of schemes. 
\item\label{higherchow} Analogous objects such as higher Chow groups may be examined. 
\end{enumerate}

The main theorem \hyperref[equmaintheorem]{\ref{equmaintheorem}} in this paper contributes to items \hyperref[higherdim]{1}, \hyperref[othercodim]{2}, \hyperref[otherinfinitesimal]{3}, \hyperref[finitechar]{5}, and \hyperref[nonsmooth]{6} above.  It contributes to item \hyperref[higherdim]{1} because it applies to $K_{p} ^{\tn{\fsz{M}}}(R,I)$ for all $p$.   It contributes to item \hyperref[othercodim]{2} because Bloch's theorem \hyperref[blochstheorem]{\ref{blochstheorem}} for a fixed $p$, applies to varieties of all dimensions.  It contributes to item \hyperref[otherinfinitesimal]{3} because it applies to a broad class of split nilpotent extensions, not merely extensions by the dual numbers.  It contributes to item \hyperref[finitechar]{5} because many rings of positive characteristic are $5$-fold stable, as noted in example \hyperref[examplesemilocalstable]{\ref{examplesemilocalstable}}.  Finally, it contributes to item \hyperref[nonsmooth]{6} because the right-hand side of Bloch's theorem \hyperref[blochstheorem]{\ref{blochstheorem}} provides one way of generalizing the Chow functors to apply to more general schemes, since the sheaves $\ms{K}_p$ are defined under very general conditions.  

The main theorem \hyperref[equmaintheorem]{\ref{equmaintheorem}}  permits interpretation of a particular class of functors on the category of smooth algebraic varieties over a field containing the rational numbers, or another appropriate category of schemes, as {\it generalized tangent functors.} These functors are given by sheafifying the isomorphism
\[K_{n+1} ^{\tn{\fsz{M}}}(R,I)\longrightarrow\frac{\Omega_{R,I} ^n}{d\Omega_{R,I} ^{n-1}},\]
of equation \hyperref[equmaintheorem]{\ref{equmaintheorem}}, and taking Zariski sheaf cohomology.  In particular, Stienstra's formal completion functor \hyperref[equstienstra]{\ref{equstienstra}} generalizes in the obvious way: 
\begin{equation}\label{equgeneralizedtangent}\widehat{\tn{Ch}}_X^n(A,\mfr{m})=H_{\tn{\fsz{Zar}}}^n\Bigg(X,\frac{\varOmega_{X\otimes_k A,X\otimes_k\mfr{m}}^n}{d\big(\varOmega_{X\otimes_k A,X\otimes_k\mfr{m}}^{n-1}\big)}\Bigg).\end{equation}
While these functors considerably broaden the picture examined by Green and Griffiths, they are almost certainly not the ``best" tangent functors available, either in the sense of information-theoretic completeness or in the sense of good formal behavior.  Their advantage lies in the relative tractability of the groups $\Omega_{R,I} ^n/d\Omega_{R,I} ^{n-1}$ compared to higher $K$-groups.   However, the ``best" generalized tangent functors can likely only be accessed by exiting the world of symbolic $K$-theory.  


\subsection*{Acknowledgements}\label{subsectionacknowledgements}

I am grateful to Wilberd Van der Kallen and Lars Hesselholt for their kind answers to several inquiries.  I would also like to thank J. W. Hoffman for making me aware of this topic.


\end{document}